\documentclass[11pt,reqno]{amsart}
\usepackage[latin1]{inputenc}
\usepackage{amssymb}
\usepackage{mathrsfs}
\usepackage{amsfonts}
\usepackage{amsfonts,amssymb,amsmath,amsthm}
\usepackage{url}
\usepackage{enumerate}
\usepackage[pdftex,bookmarksnumbered]{hyperref}
\usepackage{graphicx,xcolor}
\newcommand{\ds}{\displaystyle}
\newtheorem{theorem}{Theorem}[section]
\newtheorem{lemma}[theorem]{Lemma}
\newtheorem{proposition}[theorem]{Proposition}
\newtheorem{corollary}[theorem]{Corollary}
\theoremstyle{definition}
\newtheorem{definition}[theorem]{Definition}
\newtheorem{remark}{Remark}
\numberwithin{equation}{section}

\newtheorem{example}{Example}

\usepackage[left=2.0cm, right=2.0cm, top=2.0cm, bottom=2.0cm]{geometry}

\usepackage{graphicx}%
\usepackage{color}
\usepackage{cite}
\usepackage{hyperref}
\hypersetup{
	colorlinks = true,%
	citecolor = blue,%[rgb]{0.65,0.0,0.0},
	filecolor=red,%
	linkcolor = [rgb]{0.65,0.0,0.0},%
	anchorcolor = red,
	pagecolor = red,
	urlcolor= [rgb]{0.65,0.0,0.0}
	linktocpage=true,
	pdfpagelabels=true,
	bookmarksnumbered=true,
}

\DeclareMathOperator{\vol}{vol}

\DeclareMathOperator{\loc}{loc}

\DeclareMathOperator{\di}{div}

\DeclareMathOperator{\lo}{loc}

\DeclareMathOperator{\rr}{\mathfrak{r}}

\DeclareMathOperator{\ttt}{\mathfrak{t}}
\DeclareMathOperator{\yyy}{\mathfrak{y}}

\author{Wei Zhao}
\address{
Department of Mathematics\\
East China University of Science and Technology\\
200237 Shanghai, China}
\email{szhao\underline{ }wei@yahoo.com}

\keywords{Hardy inequality, best constant, Finsler manifold, Riemannian manifold, metric measure manifold, $p$-Laplacian, subharmonic function}
\subjclass[2010]{Primary  26D10, Secondary 53C60, 53C23}
\begin{document}

\title[Hardy inequalities with best constants on Finsler metric measure manifolds]{Hardy inequalities with best constants on Finsler metric measure manifolds}

\begin{abstract}
The paper is devoted to weighted $L^p$-Hardy inequalities with best constants on
Finsler metric measure manifolds. There are two major ingredients. The first, which is the main
part of this paper, is the Hardy inequalities concerned with distance functions in the Finsler setting. In
this case, we find that besides the flag
curvature, the Ricci curvature together with two non-Riemannian
quantities, i.e., reversibility and S-curvature, also play an important role. And we establish the optimal Hardy
inequalities  not only on noncompact manifolds, but also on closed manifolds. The second ingredient
is the Hardy inequalities for Finsler $p$-sub/superharmonic functions, in which we also investigate
the existence of extremals and the Brezis-V\'azquez improvement.
\end{abstract}
\maketitle

\section{Introduction} \label{sect1}

The classical Hardy inequality states that for any $p>1$,
\begin{align*}
\int_{\mathbb{R}^n} |\nabla u|^pdx\geq \left(\frac{|n-p|}{p}\right)^p\int_{\mathbb{R}^n} \frac{|u|^p}{|x|^p}dx,\ \forall\,u\in C^\infty_0(\mathbb{R}^n\backslash\{\mathbf{0}\}),\tag{1.1}\label{1.1newone}
\end{align*}
where
  $\left(\frac{|n-p|}{p}\right)^p$ is sharp (see for instance   Hardy et al. \cite{HPL}).
It is well-known that Hardy inequalities play a prominent role in the theory of linear and nonlinear partial differential equations. For example, they are useful to investigate the stability of solutions of semilinear elliptic and parabolic equations, the existence and asymptotic behavior of the heat equations and the stability of eigenvalues in elliptic problems. See e.g. \cite{BCC,BV,CM,DA,PV,V,VZ} and references therein.

 In recent years, a great deal of  effort  has been
devoted to the study of Hardy inequalities   in {\it curved spaces}.  As far as we know, Carron \cite{Ca} was the first who studied weighted $L^2$-Hardy inequalities  on complete, non-compact Riemannian manifolds. On one hand, inspired by \cite{Ca}, a systematic study of the Hardy inequality is carried out by Berchio,  Ganguly and  Grillo \cite{BGD}, D'Ambrosio and Dipierro \cite{DD},   Kombe and \"Ozaydin \cite{KO,KO2}, Yang, Su and Kong \cite{YSK} in the Riemannian setting (where the canonical Riemannian measure is used). On the other hand, Krist\'aly and Repov\v s  \cite{KR}, Krist\'aly and Szak\'al \cite{KS}   and Yuan, Zhao and Shen \cite{YZS} studied quantitative Hardy inequalities on Finsler manifolds with vanishing S-curvature while Mercaldo, Sano and Takahshi \cite{MST}  investigated $L^p$-Hardy inequalities in reversible Minkowski spaces (where the Busemann-Hausdorff measure is used).

In this paper, a {\it  Finsler metric measure manifold} is a Finsler manifold equipped with a smooth measure. Thus,
all the aforementioned spaces are special cases of Finsler metric measure manifolds.
However,  up to now,
limited work has been done in the study of Hardy inequalities on general Finsler metric measure manifolds. A key issue is that
 two \textit{non-Riemannian quantities}  have a great  effect on Hardy inequalities in such a setting, as
\begin{enumerate}
	\item[$\bullet$] reversibility;
		\item[$\bullet$] $S$-curvature induced by the given measure.
\end{enumerate}
 In order to emphasize the  influence, we
present a simple example.
   \begin{example}[\cite{HKZ,KRudas}]\label{Funkexam}
Consider {\it the Funk metric measure manifold} $(M,F, d\mathfrak{m}_{BH})$, where $M:=\mathbb{B}_{\mathbf{0}}(1)$ is the unit ball centered at $\mathbf{0}$ in $\mathbb{R}^n$,
\[
F(x,y):=\frac{\sqrt{|y|^2-(|x|^2|y|^2-\langle x,y\rangle^2)}}{1-|x|^2}+\frac{\langle x,y\rangle}{1-|x|^2},
\]
and $d\mathfrak{m}_{BH}$ is the Busemann-Hausdorff measure.

In this case, the Euclidean quantities $|\nabla u(x)|$, $|x|$ and $dx$ from (\ref{1.1newone}) are naturally replaced by the co-Finslerian norm of the differential $F^*(du)$, the Finsler distance function $d_F(\mathbf{0},x)$, and the measure $d\mathfrak{m}_{BH}$, respectively.
 In spite of the fact that $(M,F)$ is simply connected, forward complete and has constant flag curvature $-\frac{1}{4}$, the Hardy inequality {\it fails}:
\[
\inf_{u\in C^\infty_0(M)\backslash\{0\}}\frac{\ds \int_{M}F^{*2}{(du)}d\mathfrak{m}_{BH} }{\ds \int_{M}\frac{u^2}{d_F^2(\mathbf{0},x)}d\mathfrak{m}_{BH}}=0.\tag{1.2}\label{hard1.3funk}
\]
We remark that
$(M,F, d\mathfrak{m}_{BH})$ has \textit{infinite} reversibility and \textit{non-vanishing} $S$-curvature.
\end{example}

The purpose of this paper is  to investigate weighted $L^p$-Hardy inequalities with best constants on general Finsler metric measure manifolds.
  %Although Chern \cite{Chern} claimed that 'Finsler geometry is just Riemannian geometry without the
%quadratic restriction', subtle differences occur between these geometries.
In order to state our main results,  we introduce and recall some notations  (for details, see Section \ref{sect3}). A triple $(M,F,d\mathfrak{m})$ always denotes a  Finsler metric measure manifold, i.e.,  $(M,F)$ is a Finsler manifold endowed with a smooth measure $d\mathfrak{m}$.
Given a Finsler metric measure manifold $(M,F,d\mathfrak{m})$, the  {\it reversibility}, introduced by Rademacher \cite{R}, is defined as
\[
\lambda_F:=\sup_{x\in M} \lambda_F(x),\ \ {\rm where}\ \ \lambda_F(x)= \sup_{y\in T_xM\setminus\{0\}} \frac{F(x,-y)}{F(x,y)}.
\]
Obviously, $\lambda_F\geq 1$ with equality if and only if $F$ is {\it reversible} (i.e., symmetric). Riemannian metrics are always reversible, but there are infinitely many non-reversible Finsler metrics (e.g. see Example \ref{Funkexam}).   Furthermore, the distance function $d_F$ induced by $F$  is usually asymmetric (i.e., $d_F(x_1,x_2)\neq d_F(x_2,x_1)$) unless $\lambda_F=1$. Given a point $o\in M$,  we use the following notations throughout this paper:
\[
r_+(x):=d_F(o,x),\ r_-(x):=d_F(x,o).
\]

Since there is no canonical measure on a Finsler manifold,
 various measures can be introduced whose behavior may be genuinely different. %Two such frequently used measures are the so-called  Busemann-Hausdorff measure $d\mathfrak{m}_{BH}$ and  Holmes-Thompson measure $d\mathfrak{m}_{HT}$, see Alvarez-Paiva and  Berck\cite{AlB} and Alvarez-Paiva and Thompson \cite{AlT}. In particular, these measures for a Randers metric $F=\alpha+\beta$ are
%\[
%d\mathfrak{m}_{BH}=\left(1-\|\beta\|_\alpha^2\right)^\frac{n+1}2 dV_\alpha,\ d\mathfrak{m}_{HT}=dV_\alpha,
%\]
%where $dV_\alpha$ is the Riemannian measure induced by the Riemannian metric $\alpha$. The densities of these measures show that $d\mathfrak{m}_{BH}\leq d \mathfrak{m}_{HT}$ with equality if and only if $F$ is Riemannian (i.e., $\beta=0$).
A measure $d\mathfrak{m}$ induces
   two further  geometric quantities $\tau$ and $\mathbf{S}$,  see Shen \cite{Sh1}, which are the so-called  {\it distortion}  and {\it S-curvature}, respectively. More precisely, if $d\mathfrak{m}:=\sigma(x)dx^1\wedge...\wedge dx^n$ in some local coordinate $(x^i)$, for any $y\in T_xM\backslash\{0\}$, let
\begin{equation*}
\tau(y):=\log \frac{\sqrt{\det g_{ij}(x,y)}}{\sigma(x)},\ \ \ \  \mathbf{S}(y):=\left.\frac{d}{dt}\right|_{t=0}[\tau(\dot{\gamma}_y(t))],
\end{equation*}
where $g_y=(g_{ij}(x,y))$ is the fundamental tensor induced by $F$ and $t\mapsto \gamma_y(t)$ is the geodesic starting at $x\in M$ with $\dot{\gamma}_y(0)=y\in T_xM$. Although both the distortion and the S-curvature vanish on every Riemannian manifold   endowed with the canonical Riemannian measure,
these two quantities have already appeared in {\it Riemannian metric measure manifolds}, just in different forms.
\begin{example}\label{seconex}
Let $(M,g,e^{-f}d\vol_g)$ denote a Riemannian metric measure manifold, i.e., $(M,g)$ is a Riemannian manifold,  $f\in C^\infty(M)$ and $d\vol_g$ is the canonical Riemannian measure. Note that  $(M,g,e^{-f}d\vol_g)$ can be viewed as a Finsler metric measure space $(M,\sqrt{g},e^{-f}d\vol_g)$. Therefore,
for any $x\in M$, one has
\[
\tau(y)=f(x),\ \mathbf{S}(y)=g(y,\nabla f),\ \forall\,y\in T_xM.
\]
Thus, the S-curvature does not vanish unless $f$ is a constant.

Let $r(\cdot)$ be the distance function from a fixed point $o\in M$ induced by $g$. Then
the assumption  that {\it $\partial_rf\geq 0$ along all the minimal geodesics from $o$} is equivalent to  that {\it the S-curvature is nonnegative along all minimal geodesics  from  $o$}.
Clearly, the  Euclidean space equipped with the Gaussian measure $\left(\mathbb{R}^n,(\cdot,\cdot),\frac{1}{\sqrt{2\pi}^n}e^{-\frac12\|x\|^2}dx\right)$ satisfies this assumption if $o$ is the origin.
\end{example}

  The S-curvature  must vanish  if it is  non-positive (or nonnegative) on a reversible Finsler metric measure manifold. Hence, $\mathbf{S}\geq 0$ (or $\mathbf{S}\leq0$) is a strong condition.
   Inspired by Example \ref{seconex}, we introduce
a weaker assumption: given a point $o\in M$,
we say
 $\mathbf{S}^+_o\geq 0$ (resp., $\mathbf{S}^-_o\geq 0$) if the S-curvature is nonnegative   along all minimal geodesics {\it from} (resp., {\it to}) $o$.
And   $\mathbf{S}^\pm_o\leq 0$ are defined similarly. We remark that $\mathbf{S}^+_o\leq 0$ is not equivalent to $\mathbf{S}^-_o\geq 0$  in the irreversible case. For instance,
 the Funk metric measure manifold  in Example \ref{Funkexam} satisfies $\mathbf{S}^+_o=\mathbf{S}^-_o=\frac{n+1}2$ for every point $o\in M$.

In Finsler geometry
the \textit{flag curvature} is a geometric quantity analogous to the sectional curvature.
Let $P:=\text{Span}\{y,v\}\subset T_xM$ be a plane. The {flag curvature} is defined by
\[
\mathbf{K}(y,v):=\frac{g_y\left( R_y(v),v  \right)}{g_y(y,y)g_y(v,v)-g^2_y(y,v)},
\]
where $R_y$ is the Riemannian curvature  of $F$.   A Finsler metric measure manifold $(M,F,d\mathfrak{m})$ is called a \textit{Cartan-Hadamard measure manifold} if $(M,F)$ is a simply connected forward complete Finsler manifold  with $\mathbf{K}\leq 0.$

Our first result reads as follows.

\begin{theorem}\label{nonpositflag2}Let $(M,F,d\mathfrak{m})$ be an $n$-dimensional Cartan-Hadamard measure manifold.  Given $o\in M$, set $r_-(x):=d_F(x,o)$.  If $\mathbf{S}^-_o\geq 0$, then for any $p,\beta \in \mathbb{R}$ with $p\in (1,n)$ and $\beta>-n$,
we have
\begin{align*}
\int_M r_-^{\beta+p}\max\{F^{*p}(\pm du)\}d \mathfrak{m}\geq \left(\frac{n+\beta}{p}\right)^p\int_M r_-^\beta |u|^p d\mathfrak{m},\ \forall\,u\in C^\infty_0(M).
\end{align*}
 In particular, the constant $ \left( \frac{n+\beta}{p} \right)^p$  is sharp if $\lambda_F=1$.
\end{theorem}
If $F$ is reversible, then $F^*(\pm du)=F(\nabla u)$, i.e., the norm of the gradient of $u$.
Therefore, Theorem \ref{nonpositflag2} implies  the classical Hardy inequality  (\ref{1.1newone})  for $p\in (1,n)$, the Hardy inequality on Riemannian manifolds (cf. Yang et al. \cite[Theorem 3.1]{YSK} and  D'Ambrosio et al. \cite[Theorem 6.5]{DD}) and  the quantitative  Hardy inequality on Finsler manifolds  (cf. Krist\'aly et al. \cite{KR}). In particular, for the Funk metric measure manifold in Example \ref{Funkexam},
the inequality above yields (compare (\ref{hard1.3funk}))
\[
\inf_{u\in C^\infty_0(M)\backslash\{0\}}\frac{\ds \int_{M}\max\{F^{*2}{(\pm du)}\}d\mathfrak{m}_{BH} }{\ds \int_{M}\frac{u^2}{d_F^2(x,\mathbf{0})}d\mathfrak{m}_{BH}}\geq \left(\frac{n-2}{2}\right)^2.
\]
Furthermore,   Theorem \ref{nonpositflag2} remains valid under the weaker assumptions. See Theorem \ref{nonpositflag22} below.
The counterpart of Theorem \ref{nonpositflag2} is as follows.
\begin{theorem}\label{frirstcurHard}
Let $(M,F,d\mathfrak{m})$ be an $n$-dimensional Cartan-Hadamard measure manifold. Given $o\in M$, set $r_+(x):=d_F(o,x)$. If $\mathbf{S}^+_o\leq 0$, then for any $p,\beta\in \mathbb{R}$ with $p>n>-\beta$,
we have
\[
\int_{M}r^{\beta+p}_+\max\{ F^{*p}(\pm du) \}d\mathfrak{m}\geq \left( \frac{n+\beta}{p} \right)^p\int_{M}r_+^{\beta} {|u|^p}d\mathfrak{m},\ \forall\,u\in C^\infty_0(M).
\]

In particular, the constant $ \left( \frac{n+\beta}{p} \right)^p$  is sharp if $(M,F,d\mathfrak{m})$ satisfies
\[
\lambda_F=1, \ \mathbf{K}=0, \ \mathbf{S}^+_o=0.\tag{1.3}\label{conditionsharpcon}
\]
\end{theorem}

We discuss the spaces with (\ref{conditionsharpcon}) briefly. Since
a flat Riemannian Cartan-Hadamard  manifold is always isometric to a Euclidean space,
we get nothing new in the Riemannain setting. However, it is another story in the Finsler setting.  There are
plenty of Finsler metric measure manifolds satisfying
(\ref{conditionsharpcon})
which are not isometric to each other (see Example \ref{ConstrucMin} below).
Hence, Theorem \ref{frirstcurHard} provides a number of new models on which the inequality above is optimal.
Moreover,  this theorem   can be extended to a more general case.
 See Theorem \ref{frirstcurHard2} below.

We also have a logarithmic Hardy inequality.
\begin{theorem}\label{logone}
Let $(M,F,d\mathfrak{m})$ be an $n$-dimensional reversible  Cartan-Hadamard measure manifold.
Given $o\in M$, set $r(x):=d_F(x,o)$ and $B_o(R):=\{x\in M: r(x)<R\}$.
If $\mathbf{S}^+_o\leq0$, for any $R,p,\beta\in \mathbb{R}$ with $R>0$, $p\in (1,n]$ and $\beta<-1$,
 we have
\begin{align*}
\int_{B_o(R)}\left[\log\left(\frac{R}{r} \right) \right]^{p+\beta}F^{*p}(du) d\mathfrak{m}\geq \left(\frac{|\beta+1|}{p}\right)^p\int_{B_o(R)}\left[\log\left(\frac{R}{r} \right) \right]^\beta \frac{|u|^p}{r^p}d\mathfrak{m},
\end{align*}
for any  $u\in C^\infty_0(B_o(R))$. In particular, the constant $\left(\frac{|\beta+1|}{p}\right)^p$ is sharp.
\end{theorem}

From Theorem \ref{logone}, one can easily derive the logarithmic Hardy inequalities on Euclidean spaces, Riemannian Cartan-Hadamard manifolds and reversible Minksowski spaces, respectively (e.g. see   \cite{D,DD,MST}). Moreover,
 Theorem \ref{logone} can be generalized to the irreversible case. See Theorem \ref{logone2} below.

Now we turn to consider the Hardy inequalities concerned with the  Ricci curvature. Although this problem is also  genuinely new in the Riemannian framework, we prefer to study it in the context of Finsler geometry.

Given an $n$-dimensional Finsler metric measure manifold $(M,F, d\mathfrak{m})$,
by means of the flag curvature, one can define  the  {\it Ricci curvature} ${\bf Ric}$ in the usual way.
The {\it weighted Ricci curvature} $\mathbf{Ric}_N$, introduced in Ohta and Sturm\cite{Ot}, is defined as follows: given $N\in [n,\infty]$, for any unit vector $y\in TM$,
\begin{align*}\mathbf{Ric}_N(y)=\left\{
\begin{array}{lll}
\mathbf{Ric}(y)+\left.\frac{d}{dt}\right|_{t=0}\mathbf{S}(\gamma_y(t))-\frac{\mathbf{S}^2(y)}{N-n}, && \text{ for }N\in (n,\infty),\\
\\
\underset{L\downarrow n}{\lim}\mathbf{Ric}_L(y), && \text{ for }N=n,\\
\\
\mathbf{Ric}(y)+\left.\frac{d}{dt}\right|_{t=0}\mathbf{S}(\gamma_y(t)),  && \text{ for }N=\infty.
\end{array}
\right.
\end{align*}
%It is remarkable that bounding
%$\mathbf{Ric}_n$ from below makes sense only if $\mathbf{S}=0$.
The weighted Ricci curvature has an important influence on the geometry of Finsler manifolds. See Ohta \cite{O,O1,O2}, etc. for surveys.
 Now we   state the result as follows.

\begin{theorem} \label{Rictheno}
Let $(M,F,d\mathfrak{m})$ be an $n$-dimensional forward complete Finsler metric measure manifold with $\mathbf{Ric}_N\geq0$, where $N\in [n,\infty)$.
Given $o\in M$, define $r_+(x):=d_F(o,x)$ and $r_-(x):=d_F(x,o)$.

\smallskip

\noindent(1) Given $p,\beta\in \mathbb{R}$ with $p>N$ and $\beta<-N$, for any $u\in C^\infty_0(M\setminus\{o\})$, we have
 \[
\int_{M\backslash\{o\}}r^{\beta+p}_+\max\{ F^{*p}(\pm du) \}d\mathfrak{m}\geq \left( \frac{|N+\beta|}{p} \right)^p\int_{M\backslash\{o\}}r_+^{\beta} {|u|^p}d\mathfrak{m}.\tag{1.4}\label{5.12newnew}
\]

\smallskip

\noindent(2) Given $p,\beta\in \mathbb{R}$ with $1<p<N$ and  $\beta<-N$, for any $u\in C^\infty_0(M\setminus\{o\})$, we have
 \[
\int_{M\backslash\{o\}}r^{\beta+p}_-\max\{ F^{*p}(\pm du) \}d\mathfrak{m}\geq \left( \frac{|N+\beta|}{p} \right)^p\int_{M\backslash\{o\}}r_-^{\beta} {|u|^p}d\mathfrak{m}.\tag{1.5}\label{1.6newnew}
\]

\smallskip
\noindent In particular, the constants are sharp in (\ref{5.12newnew}) and  (\ref{1.6newnew}) if $\lambda_F=1$, $N=n$ and $p+\beta>-n$.
\end{theorem}
Clearly, the inequality (\ref{5.12newnew})  implies the  classical Hardy inequality (\ref{1.1newone}) for $p>n$.
We remark that  the manifold $M$ in Theorem \ref{Rictheno} is unnecessarily noncompact. In fact,  $M$ must be closed  if $\mathbf{Ric}_N$ is bounded below by a positive number (cf. Ohta \cite{O}).
On the other hand,
 $\mathbf{Ric}_N$ has a close relation with the {\it Bakry-\'Emery Ricci tensor}. More precisely, for a  Riemannian metric measure manifold $(M,g,e^{-f}d\vol_g)$, Example \ref{seconex} furnishes
\[
\mathbf{Ric}_{N+n}=\mathbf{Ric}_N^{\text{B\'E}}, \ \forall\,N\in (0,+\infty],
\]
where $\mathbf{Ric}_N^{\text{B\'E}}$ denotes the  $N$-Bakry-\'Emery Ricci tensor (cf. Wei and Wylie \cite{WW22}). As a consequence,
Theorem \ref{Rictheno}  inspires  the following result.

\begin{theorem}\label{cofunRic}
Let $(M,g,e^{-f}d\vol_g)$ be an $n$-dimensional closed  Riemannian metric measure manifold with $\mathbf{Ric}_{\infty}^{\text{B\'E}}\geq0$. Given $o\in M$, let $r$ denote the distance function from $o$. Suppose $\partial_r f\geq 0$ along all the minimal geodesics from $o$. Thus,
 for any $p\in (1,n)\cup(n,\infty)$ and $\beta<-n$ with $p+\beta>-n$, we have
 \[
\int_{M}r^{\beta+p}|\nabla u |^p e^{-f}d\vol_g\geq \left( \frac{|n+\beta|}{p} \right)^p\int_{M}r^{\beta} {|u|^p}e^{-f} d \vol_g, \ \forall\, u\in \mathfrak{C}^\infty(M,o),\tag{1.6}\label{1.8newnew}
\]
where $\mathfrak{C}^\infty(M,o):=\{u\in C^\infty(M): \,u(o)=0 \}$. In particular, $\left( \frac{|n+\beta|}{p} \right)^p$ is sharp with respect to $\mathfrak{C}^\infty(M,o)$.
\end{theorem}

 On a closed  manifold $M$, the Hardy inequality (\ref{1.1newone}) fails for $u\in C^\infty_0(M)$  because in this case the constant functions belong to $C^\infty_0(M)$.  Theorem \ref{cofunRic} then indicates   what kind of function  (\ref{1.1newone}) remains valid for. See Theorem \ref{reverRicinfty} below for a Finsler version of
 the theorem above.

%The benefit of using  $\max\{F^*(\pm du)\}$ in Theorem \ref{frirstcurHard} - Theorem \ref{cofunRic} is to eliminate  the reversibility  from the constant of  Hardy inequalities (see Farkas, Krist\'aly and Varga \cite[Theorem 4.1]{FKV} for example). Since these constants are sharp in the reversible case, it is a natural
 %to ask whether they remain optimal in the irreversible case? Unfortunately,

We note that Theorems \ref{nonpositflag2}-\ref{Rictheno} can be established on  Riemannian metric measure manifolds and backward complete Finsler metric measure manifolds; we leave the formulation of such statements to the interested reader.

\medskip

The paper is organized as follows. Section \ref{sect3} is devoted to preliminaries on Finsler geometry. The proofs of Theorem
\ref{nonpositflag2}-\ref{cofunRic} are given in Section \ref{distanceHardy}, while  the Hardy
inequalities for Finsler
 $p$-sub/superharmonic functions are discussed in Section \ref{Hardforpharm}. We devote Appendix \ref{Aapp} and
\ref{SObespa} to some necessary tools which are useful to prove Theorems \ref{logone} and \ref{cofunRic}.

\section{Preliminaries}\label{sect3}

\subsection{Elements from Finsler geometry} In this section, we recall some definitions and properties from Finsler geometry; for details see  Bao, Chern and Shen \cite{BCS}, Ohta and Sturm \cite{Ot} and Shen \cite{Shen2013,Sh1}.

\subsubsection{Finsler manifolds.}
%A \textit{Finsler $n$-manifold} $(M,F)$ is an $n$-dimensional differential manifold $M$ equipped with a Finsler metric $F$ which is a nonnegative function on $TM$ satisfying the following two conditions:
%
%(1) $F$ is positively homogeneous, i.e., $F(\lambda y)=\lambda F(y)$ for any $\lambda>0$ and $y\in TM$;
%
%
%(2) $F$ is smooth on $TM\backslash\{0\}$ and the Hessian $\frac{1}{2}[F^2]_{y^iy^j}(x,y)$ is positive definite, where $F(x,y):=F(y^i\frac{\partial}{\partial x^i}|_x)$.

 Let $M$ be a
$n$-dimensional connected smooth  manifold and $TM=\bigcup_{x \in M}T_{x}
M $ be its tangent bundle. The pair $(M,F)$ is a \textit{Finsler
	manifold} if the continuous function $F:TM\to [0,\infty)$ satisfies
the conditions:

\smallskip

(a) $F\in C^{\infty}(TM\setminus\{ 0 \});$

(b) $F(x,\lambda y)=\lambda F(x,y)$ for all $\lambda\geq 0$ and $(x,y)\in TM;$
%i.e., $F$ is absolutely homogeneous of degree one;

(c) $g_y:=g_{ij}(x,y)=[\frac12F^{2}%
]_{y^{i}y^{j}}(x,y)$ is positive definite for all $(x,y)\in
TM\setminus\{ 0 \}$, where $F(x,y):=F(y^i\frac{\partial}{\partial x^i}|_x)$.

\smallskip

The  quantity $g_y:=(g_{ij}(x,y))$ is called the {\it fundamental tensor}. It can be defined at $y=0$ if and only if $F$ is Riemannian, in which case $g$ is independent of $y$, i.e., $g=(g_{ij}(x))$.
The Euler theorem yields $F^2(x,y)=g_y(y,y)=g_{ij}(x,y)y^iy^j$.
Moreover, we have a Cauchy-Schwartz inequality
\[
g_y(y,w)\leq F(x,y)F(x,w), \ \forall\, y,w\in T_xM,\tag{2.1}\label{pre2.1}
\]
with equality if and only if $w=ky$ for $k\geq  0$.

Set $S_xM:=\{y\in T_xM:F(x,y)=1\}$ and $SM:=\cup_{x\in M}S_xM$. The {\it reversibility} $\lambda_F$ (cf. Rademacher \cite{R}) and the {\it uniformity constant} $\Lambda_F$ (cf. Egloff \cite{E}) of $(M,F)$ are defined as follows:
\[
\lambda_F:=\underset{y\in SM}{\sup}F(-y),\ \Lambda_F:=\underset{X,Y,Z\in SM}{\sup}\frac{g_X(Y,Y)}{g_Z(Y,Y)}.
\]
Clearly, ${\Lambda_F}\geq \lambda_F^2\geq 1$. In particular, $\lambda_F=1$ if and only if $F$ is reversible (i.e., symmetric), while $\Lambda_F=1$ if and
only if $F$ is Riemannian.  For convenience, we introduce the reversibility of a subset $U\subset M$, i.e.,
\[
\lambda_F(U):=\sup_{y\in SU}F(-y), \text{ where }SU:=\cup_{x\in U}S_xM.
\]
Thus, $\lambda_F(M)=\lambda_F$, and $\lambda_F(U)$ is finite if $\overline{U}$ is compact.

The {\it dual Finsler metric} $F^*$ of $F$ on $M$ is
defined by
\begin{equation*}
F^*(x,\xi):=\underset{y\in T_xM\backslash \{0\}}{\sup}\frac{\xi(y)}{F(x,y)}, \ \
\forall \xi\in T_x^*M,\tag{2.2}\label{2.2newineq}
\end{equation*}
which is   a Finsler metric on $T^*M$. Let $g_\xi^*$ be the fundamental tensor of $F^*$.  Then Yuan et al. \cite[Theorem 3.5]{YZS} furnishes
\[
F^{*2}(\xi+\eta)\geq F^{*2}(\xi)+2g^*_\xi(\xi,\eta)+\frac{1}{\Lambda_F}F^{*2}(\eta),\ \forall\, \xi,\eta\in T^*_xM,\tag{2.3}\label{ineq}
\]
where  $g^*_\xi(\xi,\eta):=0$ if $\xi=0$.

The {\it Legendre transformation} $\mathfrak{L} : TM \rightarrow T^*M$ is defined
by
\begin{equation*}
\mathfrak{L}(X):=\left \{
\begin{array}{lll}
 g_X(X,\cdot), & \ \ \text{ if } X\neq0, \\
 \\
0, & \ \ \text{ if } X=0.%
\end{array}
\right.
\end{equation*}
In particular,  $\mathfrak{L}:TM\backslash\{0\}\rightarrow T^*M\backslash\{0\}$ is a diffeomorphism with $F^*(\mathfrak{L}(X))=F(X)$, for any $X\in TM$.
Now let $f : M \rightarrow \mathbb{R}$ be a $C^1$-function on $M$; the
{\it gradient} of $f$ is defined as $\nabla f = \mathfrak{L}^{-1}(df)$. Thus,  $df(X) = g_{\nabla f} (\nabla f,X)$. For a non-Riemannian Finsler metric, $\nabla$ is usually nonlinear, i.e., $\nabla(f+h)\neq\nabla f+\nabla h$.

Let $\zeta:[0,1]\rightarrow M$ be a Lipschitz continuous path. The length of $\zeta$ is defined by
\[
L_F(\zeta):=\int^1_0 F(\dot{\zeta}(t))dt.
\]
Define the {\it distance function} $d_F:M\times M\rightarrow [0,+\infty)$ by
$d_F(x_1,x_2):=\inf L_F(\sigma)$,
where the infimum is taken over all
Lipschitz continuous paths $\zeta:[a,b]\rightarrow M$ with
$\zeta(a)=x_1$ and $\zeta(b)=x_2$. Generally,  $d_F(x_1,x_2)\neq d_F(x_2,x_1)$ unless $F$ is reversible.
The {\it forward and backward metric balls} $B^+_o(R)$ and $B^-_o(R)$ are defined by
\[
B^+_o(R):=\{x\in M:\, d_F(o,x)<R\},\ B^-_o(R):=\{x\in M:\, d_F(x,o)<R\}.
\]
If $F$ is reversible, forward metric balls coincide with backward ones, which are denoted by $B_o(R)$.

Given $o\in M$, set $r_+(x):=d_F(o,x)$ and $r_-(x):=d_F(x,o)$.    Shen \cite[Lemma 3.2.3]{Sh1}  yields
\[
F(\nabla r_+)=F^*(dr_+)=1,\ F(\nabla (-r_-))=F^*(-dr_-)=1,
\]
 a.e. on $M$. If $F$ is reversible,  both $r_\pm(x)$ are denoted by $r(x)$.

A smooth curve $t\mapsto \gamma(t)$ in $M$ is called a (constant speed) \textit{geodesic} if it satisfies
\[
\frac{d^2\gamma^i}{dt^2}+2G^i\left(\frac{d\gamma}{dt}\right)=0,
\]
where
\begin{align*}
G^i(y):=\frac14 g^{il}(y)\left\{2\frac{\partial g_{jl}}{\partial x^k}(y)-\frac{\partial g_{jk}}{\partial x^l}(y)\right\}y^jy^k\tag{2.4}\label{geoedesiccon}
\end{align*}
is the geodesic coefficient.
And we always use $\gamma_y(t)$ to denote  the geodesic with $\dot{\gamma}_y(0)=y$.

The Finsler manifold $(M,F)$ is {\it forward complete} if  every geodesic $t\mapsto \gamma(t)$, $0\leq t<1$, can be extended to a geodesic defined on $0\leq t<\infty$; similarly,  $(M,F)$ is  {\it backward complete} if  every geodesic $t\mapsto \gamma(t)$, $0< t\leq 1$, can be extended to a geodesic defined on $-\infty< t\leq 1$. If $(M,F)$ is both forward complete and backward complete, we say $(M,F)$ is {\it complete} for short.

 The \textit{cut value} $i_y$ of $y\in S_xM$ is defined by
\[
i_y:=\sup\{t: \text{ the geodesic }\gamma_y|_{[0,t]} \text{ is globally minimizing}  \}.
\]
The \textit{injectivity radius} at $x$ is defined as $\mathfrak{i}_x:=\inf_{y\in S_xM} i_y$.  According to Bao et al. \cite{BCS} and Yuan et al. \cite[Proposition 3.2]{YZS}, if $(M,{F})$ is either forward or backward complete, then ${\mathfrak{i}_x}>0$ for any point $x\in M$.
The \textit{cut locus} of $x$ is defined as
\[
\text{Cut}_x:=\left\{\gamma_y(i_y):\,y\in S_xM \text{ with }i_y<\infty \right\}.
\]
In particular, $\text{Cut}_x$ is closed and has null measure.

\subsubsection{Measures and curvatures}A triple $(M,F,d\mathfrak{m})$ is called a {\it FMMM}, (i.e., {\it Finsler metric measure manifold})
if $(M,F)$ is a Finsler  manifold endowed with a smooth measure $d\mathfrak{m}$.  In the sequel, the function $x\mapsto \sigma(x)$ denotes the {\it density function} of  $d\mathfrak{m}$ in a local coordinate system $(x^i)$, i.e.,
\[
d\mathfrak{m}=:\sigma(x)dx^1\wedge \cdots\wedge  dx^n.\tag{2.5}\label{measure}
\]

The {\it divergence} of a vector filed $X$ is defined as
\[
\di(X)\, d\mathfrak{m}:=d\left( X\rfloor d\mathfrak{m}\right).
\]
If $M$ is compact and oriented, we have the  divergence theorem
\[
\ds\int_M\di(X)d\mathfrak{m}=\ds\int_{\partial M} g_{\mathbf{n}}(\mathbf{n},X)\,d A,\tag{2.6}\label{olddivlem}
\]
where $dA=\mathbf{n}\rfloor d\mathfrak{m}$, and $\mathbf{n}$ is the unit outward normal vector field on $\partial M$, i.e., $F(\mathbf{n})=1$ and $ g_{\mathbf{n}}(\mathbf{n},Y)=0$ for any $Y\in T(\partial M)$.

Given a $C^2$-function $f$, set $\mathcal {U}=\{x\in M:\, df|_x\neq0\}$. The \textit{Laplacian} of $f$ is defined on $\mathcal {U}$ by
\begin{align*}
\Delta f:=\text{div}(\nabla f)=\frac{1}{\sigma(x)}\frac{\partial}{\partial x^i}\left(\sigma(x)g^{*ij}(df|_x)\frac{\partial f}{\partial x^j}\right),
\end{align*}
where $\sigma(x)$ is defined by (\ref{measure}) and $(g^{*ij})$ is the fundamental tensor of $F^*$.  As in Ohta et al. \cite{Ot}, we define
the distributional Laplacian of $u\in W^{1,2}_{\text{loc}}(M)$
in the weak sense by
\[
\ds\int_M v{\Delta} u d\mathfrak{m}=-\ds\int_M\langle\nabla u, dv\rangle d\mathfrak{m} \text{ for all }v\in C^\infty_0(M),
\]
where $\langle \nabla u, dv\rangle:= dv(\nabla u)$ at $x\in M$ denotes the canonical pairing between $T_xM$ and $T^*_xM.$

By (\ref{measure}),  the {\it distortion} $\tau$ and the {\it S-curvature} $\mathbf{S}$  of $(M,F,d\mathfrak{m})$ are defined  as
\begin{equation*}
\tau(y):=\log \frac{\sqrt{\det g_{ij}(x,y)}}{\sigma(x)}, \ \mathbf{S}(y):=\left.\frac{d}{dt}\right|_{t=0}[\tau(\dot{\gamma}_y(t))],\text{ for $y\in T_xM\backslash\{0\}$},
\end{equation*}
where $\sigma(x)$ is defined by (\ref{measure}) and $\gamma_y(t)$ is a  geodesic with $\dot{\gamma}_y(0)=y$.

 Given a point $o\in M$, we say $\mathbf{S}_o^+\geq 0$ (resp., $\mathbf{S}_o^+\leq 0$) if the S-curvature is nonnegative (resp., non-positive) along every minimal geodesic {\it from} $o$. On the other hand, we say
$\mathbf{S}_o^-\geq 0$ (resp., $\mathbf{S}_o^-\leq 0$) if the S-curvature is nonnegative (resp., non-positive) along every minimal geodesic {\it to} $o$. In particular, if $F$ is reversible, then $\mathbf{S}_o^+\geq 0$ (resp., $\mathbf{S}_o^+\leq 0$) if and only if
$\mathbf{S}_o^-\leq 0$ (resp., $\mathbf{S}_o^-\geq 0$).

The {\it Riemannian curvature} $R_y$ of $F$ is a family of linear transformations on tangent spaces. More precisely, set
$R_y:=R^i_k(y)\frac{\partial}{\partial x^i}\otimes dx^k$, where
\begin{align*}
R^i_{\,k}(y)&:=2\frac{\partial G^i}{\partial x^k}-y^j\frac{\partial^2G^i}{\partial x^j\partial y^k}+2G^j\frac{\partial^2 G^i}{\partial y^j \partial y^k}-\frac{\partial G^i}{\partial y^j}\frac{\partial G^j}{\partial y^k},
\end{align*}
and $G^i$'s are the geodesic constants defined in (\ref{geoedesiccon}).

Let $P:=\text{Span}\{y,v\}\subset T_xM$ be a plane. The {flag curvature} is defined by
\[
\mathbf{K}(y,v):=\frac{g_y\left( R_y(v),v  \right)}{g_y(y,y)g_y(v,v)-g^2_y(y,v)}.
\]

The  {\it Ricci curvature} at $y\neq 0$
is defined by
$\mathbf{Ric}(y):= \frac{R^i_{i}(y)}{F^2(y)}$. According to Ohta et al. \cite{Ot}, given $y\in SM$,
 the {\it weighted Ricci curvature} is defined by
\begin{align*}\mathbf{Ric}_N(y)=\left\{
\begin{array}{lll}
\mathbf{Ric}(y)+\left.\frac{d}{dt}\right|_{t=0}\mathbf{S}(\gamma_y(t))-\frac{\mathbf{S}^2(y)}{N-n}, && \text{ for }N\in (n,\infty),\\
\\
\underset{L\downarrow n}{\lim}\mathbf{Ric}_L(y), && \text{ for }N=n,\\
\\
\mathbf{Ric}(y)+\left.\frac{d}{dt}\right|_{t=0}\mathbf{S}(\gamma_y(t)),  && \text{ for }N=\infty.
\end{array}
\right.
\end{align*}
In particular,  bounding
$\mathbf{Ric}_n$ from below makes sense only if $\mathbf{S}=0$.

If $(M,F)$ is either forward complete or backward complete, then  there exists a  polar coordinate system around every point in $M$ (cf.\,Yuan et al. \cite[Proposition 3.2]{YZS} and Zhao et al. \cite[Section 3]{ZS}). Fixing an arbitrary point $o\in M$,
let $(t,y)$ denote the {\it polar coordinate system} around $o$ and write
\[
d\mathfrak{m}=: \hat{\sigma}_o(t,y)dt\wedge d\nu_o(y),\tag{2.7}\label{formulavolume}
\]
where $t:=r_+$ and $d\nu_o(y)$ is the Riemannian volume measure on $S_oM$ induced by $F$.

Since $S_oM$ is compact,
the integral $\int_{S_oM}e^{-\tau(y)}d\nu_o(y)$ is  finite. Particularly, this integral is equal to the volume of the standard $(n-1)$-unit Euclidean sphere if $d\mathfrak{m}$ is the Busemann-Hausdorff measure (cf. Shen \cite{Shen_Adv_Math} or Zhao et al. \cite{ZS}).

For any fixed $y\in S_oM$, we have
\[
\Delta t=\frac{\partial}{\partial t}\log( \hat{\sigma}_o(t,y)), \text{ for }0<t<i_y.
\]
In particular, Zhao et al. \cite[Lemma 3.1]{ZS} yields
\[
\lim_{r\rightarrow 0^+}\frac{ \hat{\sigma}_o(t,y)}{t^{n-1}}=e^{-\tau(y)}.\tag{2.8}\label{v14-2.1}
\]
According to Zhao et al. \cite[Theorem 4.3,\,Remark 5.3,\,Theorem 3.6]{ZS}, if $\mathbf{Ric}(\nabla t)\geq -(n-1)k^2$ and $\mathbf{S}(\nabla t)\geq -h^2$, then
\[
 \hat{\sigma}_o(t,y)\leq e^{-\tau(y)+h^2t}\mathfrak{s}^{n-1}_{-k^2}(t), \text{ for any }y\in S_oM,\ 0<t<i_y.\tag{2.9}\label{Riccompar2.8}
\]
where $\mathfrak{s}_{-k^2}(t)$ is the unique solution to the equation $f''(t)-k^2f(t)=0$ with $f(0)=0$, $f'(0)=1$.

%On the other hand, due to \cite{Sh1}, if $\mathbf{K}\leq 0$ and $\mathbf{S}\leq 0$, then  for any $y\in S_oM$, it holds
%\[
%\Delta t\geq \frac{n-1}{t},\ 0<t<i_y.\tag{2.9}\label{v14-2.3}
%\]
\subsubsection{Reverse Finsler metric measure manifolds}
Given a FMMM $(M,F,d\mathfrak{m})$,
according to  Ohta et al. \cite{Ot}, the {\it reverse} of $F$ is defined by $\overleftarrow{F}(x,y):=F(x,-y)$, which is also a Finsler metric.  Clearly, $(M,F)$ is forward (resp., backward) complete if and only if $(M,\overleftarrow{F})$ is backward (resp., forward)  complete.

In this paper,
$(M,\overleftarrow{F},d\mathfrak{m})$ is called the {\it RFMMM} (i.e., {\it reverse Finsler metric measure manifold}).
Let $\overleftarrow{*}$  denote the geometric quantity $*$ defined by $\overleftarrow{F}$. Then we have
\begin{align*}\left\{
\begin{array}{lll}
\overleftarrow{r}_+(x)=r_-(x),\ \overleftarrow{r}_-(x)= r_+(x),&& \text{ for any }x\in M;\\
\\
\overleftarrow{\mathbf{K}}(y,v)=\mathbf{K}(-y,v),\ \overleftarrow{\mathbf{Ric}}(y)=\mathbf{Ric}(-y),\overleftarrow{\mathbf{S}}(y)=-\mathbf{S}(-y), && \text{ for any }y,v\in TM\backslash\{0\};\tag{2.10}\label{revesequn}\\
\\
\overleftarrow{\nabla} f=-\nabla (-f),\ \overleftarrow{\Delta}f=-\Delta (-f),  && \text{ for any }f\in C^2(M).
\end{array}
\right.
\end{align*}

\begin{remark}\label{Scurvature}One can use the polar coordinates to describe  of $\mathbf{S}^\pm_o$. More precisely,
let $(t,y)$ be the polar coordinate system around $o$ in $(M,F, d\mathfrak{m})$. Thus, one has
\[
\mathbf{S}_o^+\geq 0 \text{ (resp., $\leq 0$)} \Longleftrightarrow \mathbf{S}(\nabla t)\geq 0 \text{ (resp., $\leq 0$) }\text{ for any }y\in S_oM, \ 0<t<i_y.
\]
On the other hand, let $(\mathfrak{t},\mathfrak{y})$ denote the polar coordinate system around $o$ in $(M,\overleftarrow{F}, d\mathfrak{m})$. Then
\[
\mathbf{S}_o^-\geq 0 \text{ (resp., $\leq 0$)}  \Longleftrightarrow \overleftarrow{\mathbf{S}}(\overleftarrow{\nabla} \mathfrak{t})\leq 0 \text{ (resp., $\geq 0$) }\text{ for any }\mathfrak{y}\in \overleftarrow{S_oM}, \ 0<\mathfrak{t}<\overleftarrow{i_{\mathfrak{y}}}.
\]
\end{remark}

\section{Hardy inequalities for  distance functions}\label{distanceHardy}
In this section, we study the Hardy inequalities concerned with distance functions and show Theorem \ref{frirstcurHard}-Theorem \ref{cofunRic}. Our approach  is mainly based on a generalization of the divergence
theorem in D'Ambrosio \cite{D} together with the sharp volume comparison for arbitrary measures in Zhao et al. \cite{ZS}. For simplicity of presentation, we introduce some notations, which are used throughout this paper.

\smallskip

\noindent \textbf{Notations:}
(1) Let $\Omega$ be a domain (i.e., a connected open subset) in a forward complete Finsler manifold $(M,F)$. We say that {\it $\Omega$ is a natural domain} if one of the following statements holds:

\smallskip

(i) $\Omega\subset M$ is a proper domain with smooth non-empty boundary;

\smallskip

(ii) $\Omega=M$ if $M$ is noncompact.

\smallskip

\noindent(2) Let $\Omega$ be a natural domain in a forward complete FMMM $(M,F, d\mathfrak{m})$.
We say that a vector filed $X$ belongs to $L^1_{\lo}(T\Omega)$ if  $\int_K F(X)d\mathfrak{m}$ is finite for any compact set $K\subset \Omega$.
Given a vector filed $X\in L^1_{\lo}(T\Omega)$ and a nonnegative function  $f_X\in L^1_{\lo}(\Omega)$, we say that {\it $f_X\leq \di X$ in the weak sense} if
\[
\int_\Omega u f_X d\mathfrak{m}\leq -\int_\Omega \langle  X, du \rangle d\mathfrak{m}, \ \forall \,u\in C^1_0(\Omega) \text{ with }u\geq 0,
\]
where $\langle  X,du \rangle:=du(X)$.

\smallskip

\noindent(3) A quadruple $(M,o,F,d\mathfrak{m})$ is called a {\it PFMMM} (i.e., {\it pointed Finsler metric measure manifold}) if $(M,F,d\mathfrak{m})$ is a Finsler metric measure manifold and $o$ is a point in $M$. In such a space, $r_+$ (resp., $r_-$) is always defined as the distance from $o$ (resp., to $o$), i.e., $r_+(x)=d_F(o,x)$ (resp., $r_-(x)=d_F(x,o)$). In particular, $r_\pm$ are denoted by $r$ if $F$ is reversible.

\smallskip

\subsection{Main tools}
In this subsection, we present the main tools.
Inspired by D'Ambrosio\cite{D}, we first establish a divergence theorem.

\begin{theorem}\label{divlemf}Let $\Omega$ be a natural domain in a forward complete FMMM $(M,F, d\mathfrak{m})$.
Let $X\in L^1_{\lo}(T\Omega)$ be a vector filed  and let $f_X\in L^1_{\lo}(\Omega)$ be a nonnegative function.
Given $p>1$, suppose the following conditions hold:

\smallskip

\ \ \ \ \ (i) $f_X\leq \di X$ in the weak sense;
\ \ \ \ \ (ii) $F^p(X)/f_X^{p-1}\in L^1_{\lo}(\Omega)$.

\smallskip

\noindent Then  we have
\begin{align*}
(1) &&p^p {\int_\Omega\frac{F^p(-X)}{f_X^{p-1}}\max\{F^{*p}(\pm d u)\}d\mathfrak{m}}    \geq{\int_\Omega |u|^p f_Xd\mathfrak{m}},\ \forall\,u\in C^\infty_0(\Omega),\\
(2)&& p^p {\int_\Omega\frac{F^p(X)}{f_X^{p-1}}\max\{F^{*p}(\pm d u)\}d\mathfrak{m}}    \geq{\int_\Omega |u|^p f_Xd\mathfrak{m}},\ \forall\,u\in C^\infty_0(\Omega).
\end{align*}
\end{theorem}
\begin{proof} Since the reversibility $\lambda_F(K)$ is finite for any compact set $K$, Condition (ii) implies $F^p(-X)/f_X^{p-1}\in L^1_{\lo}(\Omega)$.  Now we show (1).
It is easy to check $F^*(udu)\leq |u|\max\{F^*(\pm du)\}$, which together with the assumption, (\ref{2.2newineq}) and
the H\"older inequality   yields
\begin{align*}
&\int_\Omega |u|^p f_Xd\mathfrak{m}{\leq}-\int_\Omega \langle X, d(u^2)^{p/2} \rangle d\mathfrak{m}=\int_\Omega p|u|^{p-2}\langle -X, u du \rangle d\mathfrak{m}\\
\leq& p\int_\Omega |u|^{p-2} F(-X)F^*(u du)d\mathfrak{m}\leq p\int_\Omega |u|^{p-1}F(-X)\max\{F^*(\pm du)\}d\mathfrak{m}\\
{\leq}& p\left( \int_\Omega|u|^p f_Xd\mathfrak{m} \right)^{\frac{p-1}p}\left(\int_\Omega \frac{F(-X)^p}{f_X^{p-1}}\max\{F^{*p}(\pm du)\}d\mathfrak{m} \right)^{\frac{1}p}.
\end{align*}
Hence, (1) follows. In order to prove (2),  note
$\langle -X, u du \rangle\leq F(X)F^*(-u d u)\leq |u|F(X)\max\{F^*(\pm du)\}$.
Then the rest of the proof is the same as above.
\end{proof}

The following result  also plays an important role in establishing the Hardy inequalities.
\begin{lemma}\label{centerinteg}
Let $(M,o,F,d\mathfrak{m})$ be an $n$-dimensional forward or backward complete PFMMM. Let $\mathfrak{r}$ denote either $r_+$ or $r_-$ and set $\mathfrak{B}_o(s):=\{x\in M:\,\rr(x)<s\}$. Then

\smallskip

(1) For any
$k\in (-\infty,n)$, we have
\[
\ds\lim_{\epsilon\rightarrow 0^+}\int_{\mathfrak{B}_o(\epsilon)\backslash\{o\}} {{\rr}^{-k} }d\mathfrak{m}=0.
\]

\smallskip

(2)
For any $R>0$, we have
\[
\ds\lim_{\epsilon\rightarrow 0^+}\int_{\mathfrak{B}_o(R)\backslash \mathfrak{B}_o(\epsilon)} {\rr}^{-n} d\mathfrak{m}=+\infty.
\]
\end{lemma}
\begin{proof}
(i) Suppose $\mathfrak{r}=r_+$. Let $(t,y)$ be the polar coordinate system around $o$. In view of (\ref{v14-2.1}), there exists  an $\epsilon_0\in (0, \mathfrak{i}_o)$ such that for any $t\in (0,\epsilon_0)$,
\[
\frac12 e^{-\tau(y)}t^{n-1}\leq \hat{\sigma}_o(t,y)\leq 2e^{-\tau(y)}t^{n-1},\ \forall\,y\in S_oM.
\]
Therefore, if $k\in (-\infty,n)$, (\ref{formulavolume}) together with the inequality above furnishes
\begin{align*}
\int_{B^+_o(\epsilon)\backslash\{o\}}r^{-k}_+d\mathfrak{m}= \int_{S_oM}d\nu_o(y)\int_0^\epsilon t^{-k}\hat{\sigma}_o(t,y)dt\leq \frac{2\epsilon^{n-k}}{n-k}\int_{S_oM}e^{-\tau(y)}d\nu_o(y)\rightarrow0, \text{ as }\epsilon\rightarrow0^+.
\end{align*}
Similarly, we obtain
\begin{align*}
\int_{B^+_o(R)\backslash B^+_o(\epsilon)}r^{-n}_+d\mathfrak{m}
\geq\frac12\ln\left( \frac{\min\{R,\mathfrak{i}_o\}}{\epsilon} \right)\int_{S_oM}e^{-\tau(y)}d\nu_o(y)\rightarrow+\infty, \text{ as }\epsilon\rightarrow0^+.
\end{align*}

(ii) Suppose $\mathfrak{r}=r_-$. In this case, we consider the   RFMMM $(M,\overleftarrow{F},d\mathfrak{m})$.
Thus, the above results hold for $\overleftarrow{r_+}(x):=d_{\overleftarrow{F}}(o,x)$. The assertions  then follow from  $r_-(x)=\overleftarrow{r_+}(x)$.
\end{proof}

\begin{definition}
Let $\Omega\subset M$ be a natural domain in a FMMM  $(M,F,d\mathfrak{m})$.
Given $p> 1$, for an arbitrary function $f\in C^\infty(\Omega)$,
the {\it $p$-Laplacian} of $f$ is defined  as
 \[
 \Delta_p f:=\di\left(F^{p-2}(\nabla f)\, \nabla f\right) \text{ on  }\,\mathcal {U}:=\{x\in M: df|_x\neq 0\}.
\]
%If $f\in W^{1,p}_{\lo}(\Omega)$, the $p$-Laplacian of $f$ in the distribution sense is
%\[
%\int_M u \Delta_p f d\mathfrak{m}:=-\int_M F^{p-2}(\nabla f)\langle\nabla f,du\rangle d\mathfrak{m},  \ \forall \,u\in C^\infty_0(\Omega).
%\]

%\smallskip

\noindent Given $c\in \mathbb{R}$,
we say that a function $\rho(x)\in W^{1,p}_{\lo}(\Omega)$ satisfies {\it $-c\Delta_p\rho\geq 0$ in the weak sense} if
\[
c\int_\Omega  F^{p-2}(\nabla \rho)\langle\nabla \rho,du\rangle d\mathfrak{m}\geq 0, \ \forall \,u\in C^1_0(\Omega) \text{ with }u\geq 0.
\]
\end{definition}

%\begin{remark}\label{imprtanremark}\textbf{Let $(M,F)$ be a forward complete Finsler manifold.
%The test function of sharpness follows from the theorem: For any
 %$\Omega\subset M$, $C^\infty_c(\Omega)$ is dense in $C_c(\Omega
%)$ under the
%uniform  topology. That is, for any $f\in C_c(\Omega
%)$, there exists an sequence of $f_n\in C^\infty_c(\Omega)$ such that $f_n\rightrightarrows f$ on $\Omega$ with $d_H(\text{supp}f_n,\text{supp}f)<1/n$.
%In particular, $\Omega=M$. Hence, in order to prove the sharpness, we only need to find a continuous test function with compact support.}
%\end{remark}

\begin{lemma}\label{mainlemmforcr}
Let $(M,o,F,d\mathfrak{m})$ be a  forward complete PFMMM and let $\Omega\subset M$ be a natural domain.

\smallskip

\noindent (i) Suppose that $\alpha>0$ and $\beta\in \mathbb{R}$ satisfy the following conditions:

\smallskip

 (1) $r_+^{(\alpha-1)(p-1)},r_+^\beta,r_+^{\beta+p}\in L^1_{\lo}(\Omega)$;

 (2) $-c\Delta_p(r_+^\alpha)\geq 0$ in the weak sense, where $c=\alpha[(\alpha-1)(p-1)-\beta-1]$.

\smallskip

\noindent Then we have
\begin{align*}
\int_{\Omega} r_+^{\beta+p}\max\{F^{*p}(\pm du)\}d \mathfrak{m}\geq (\vartheta_{\alpha,\beta,p})^p\int_{\Omega} r_+^\beta |u|^p d\mathfrak{m},\ \forall\,u\in C^\infty_0(\Omega),\tag{3.1}\label{5.1.1}
\end{align*}
 where
 \[
 \vartheta_{\alpha,\beta,p}:=\frac{|(\alpha-1)(p-1)-\beta-1|}{p}.
 \]

\smallskip

\noindent (ii) Suppose that $\alpha>0$ and $\beta\in \mathbb{R}$ satisfy the following conditions:

\smallskip

 (1') $r_-^{(-\alpha-1)(p-1)},r_-^\beta,r_-^{\beta+p}\in L^1_{\lo}(\Omega)$;

 (2') $-c\Delta_p(r_-^{-\alpha})\geq 0$ in the weak sense, where $c=\alpha[(\alpha+1)(p-1)+\beta+1]$.

\smallskip

\noindent Then we have
\begin{align*}
\int_{\Omega} r_-^{\beta+p}\max\{F^{*p}(\pm du)\}d \mathfrak{m}\geq (\vartheta_{-\alpha,\beta,p})^p\int_{\Omega} r_-^\beta |u|^p d\mathfrak{m},\ \forall\,u\in C^\infty_0(\Omega),\tag{3.2}\label{5.1.2}
\end{align*}
 where
 \[
 \vartheta_{-\alpha,\beta,p}:=\frac{|(\alpha+1)(p-1)+\beta+1|}{p}.
 \]
\end{lemma}

\begin{proof}
 (i) Provided that $-\Delta_p (r_+^\alpha)\geq 0$ and $c>0$, we set
\[
X:=-\alpha r_+^{\beta+1}\nabla r_+,\ f_X:=c r_+^\beta.
\]
Clearly, $f_X\in L^1_{\lo}(\Omega)$. The H\"older inequality together with Condition (1) implies $r^{\beta+1}_+\in L^1_{\lo}(\Omega)$ and hence,
$X\in L^1_{\lo}(T\Omega)$.
Moreover,  $F^{p}(X)/f^{p-1}_X\in L^1_{\lo}(\Omega)$ because $r_+^{\beta+p}\in L^1_{\lo}(\Omega)$.

 Given $\epsilon>0$ and $u\in C^1_0(\Omega)$ with $u\geq 0$, set $r_\epsilon:=\epsilon+r_+$ and $ v:=r_\epsilon^{-\frac{c}a}u$.
  Since $-\Delta_p(r_+^\alpha)\geq 0$ and  $r_+^{(\alpha-1)(p-1)}\in L^1_{\lo}(\Omega)$, we have $\int_{\Omega} \langle F^{p-2}(\nabla r_+^\alpha)\nabla r_+^\alpha, dv\rangle d \mathfrak{m}\geq 0$, that is,
\begin{align*}
c\alpha^{p-2}\int_{\Omega}  r_+^{(\alpha-1)(p-1)}u r_\epsilon^{-\frac{c}a-1}d\mathfrak{m}\leq \alpha^{p-1}\int_{\Omega} r_+^{(\alpha-1)(p-1)}r_\epsilon^{-\frac{c}a}\langle \nabla r_+,du\rangle d\mathfrak{m}.\tag{3.3}\label{5.3newneed}
\end{align*}
Since
$
r_+^{(\alpha-1)(p-1)}r_\epsilon^{-\frac{c}a-1}\leq r_+^\beta\in L^1_{\lo}(\Omega)$ and $r_+^{(\alpha-1)(p-1)}r_\epsilon^{-\frac{c}a}\leq r_+^{\beta+1}\in L^1_{\lo}(\Omega)$,
the  Lebesgue's dominated convergence theorem together with (\ref{5.3newneed}) yields
\[
\int_{\Omega} cr^{\beta}_+ud\mathfrak{m}\leq\int_{\Omega}\alpha r^{\beta+1}_+\langle \nabla r_+,du\rangle d\mathfrak{m},
\]
that is, $f_X\leq \di X$ in the weak sense. Now (\ref{5.1.1}) follows from Theorem \ref{divlemf} (1).

In the case when $-\Delta_p (r^\alpha_+)\leq 0$ and $c<0$, set $X:=\alpha r_+^{\beta+1}\nabla r_+,\ f_X:=-c r_+^\beta$ and $v:=r_+^{-c/\alpha}u$. Then (\ref{5.1.1}) follows from
a similar argument and Theorem \ref{divlemf} (2).

(ii) In order to show  (\ref{5.1.2}), one set
\begin{align*}\left\{
\begin{array}{lll}
X:=-\alpha r^{\beta+1}\nabla(-r_-),\ f_X:=cr^\beta_-,\ v:=r_-^{\frac{c}a}u, && \text{if }c>0\text{ and } -\Delta_p(r_-^{-\alpha})\geq0,\\
\\
X:=\alpha r^{\beta+1}\nabla(-r_-), \ f_X:=-cr^\beta_-,\ v:=(r_-+\epsilon)^{\frac{c}a}u, && \text{if }c<0\text{ and } -\Delta_p(r_-^{-\alpha})\leq0.
\end{array}
\right.
\end{align*}
Then the proof follows in a similar manner and hence, we omit it.
\end{proof}

We introduce  the following space  to investigate the sharpness of constants of Hardy inequalities.

\begin{definition}\label{DefDS}
Let $(M,o,F,d\mathfrak{m})$ be an $n$-dimensional complete reversible PFMMM and let $\Omega\subset M$ be a natural domain.  Given $p>1$ and $\beta\in \mathbb{R}$, suppose $o\notin \Omega$ if $\beta\leq -n$.
Denote by $D^{1,p}(\Omega,r^{\beta+p})$  the closure of  $C^\infty_0(\Omega)$ with respect to the norm
\[
\|u\|_D:=\left( \int_{\Omega}r^{\beta} {|u|^p}d\mathfrak{m}+\int_\Omega r^{p+\beta}F^{*p}(d{u})d \mathfrak{m}\right)^{\frac1p},\tag{3.4}\label{norm3.4}
\]
where $r(x):=d_F(o,x)$.
\end{definition}
Lemma \ref{centerinteg} implies that both $\|\cdot\|_D$ and  $D^{1,p}(\Omega,r^{\beta+p})$ are well-defined.
The following lemma provides the extremal functions for some Hardy inequalities.
\begin{lemma}\label{impsharplem}
Let $(M,o,F,d\mathfrak{m})$ be a  complete reversible PFMMM with $\mathbf{S}_o^+\geq 0$ and either $\mathbf{Ric}\geq0$ or $\mathbf{Ric}_\infty\geq 0$. Suppose that $p,\beta\in \mathbb{R}$ and $\Omega$ satisfy one of the following conditions:

\smallskip

(1) $\beta>-n$, $p>1$ and $\Omega:=M$ is noncompact;
\ \ \ (2) $\beta<-n$, $p>\max\{1,-n-\beta\}$ and $\Omega:=M\backslash\{o\}$.

\smallskip

\noindent For any  $\delta\in (0,1)$ and any $s\in (0,\min\{1, \mathfrak{i}_o/2\})$, set $c(\delta):= |n+\beta|/p+\delta/p$ and
\begin{align*}v(x):=\left\{
\begin{array}{lll}
\left(\frac{r(x)}{s}\right)^{c(\delta)},\ \text{ if  }x\in B_o(s),\\
\\
\left(\frac{r(x)}{s}\right)^{-c(\delta/2)}, \ \text{ if  }x\in M\setminus {B_o(s)}.
\end{array}
\right.
\end{align*}
Thus, $v\in D^{1,p}(\Omega,r^{p+\beta})$. In particular,
\[
c^p(\delta)\int_\Omega r^\beta |v|^p  d\mathfrak{m}>\int_\Omega r^{\beta+p} F^{*p}(dv)d\mathfrak{m}.\tag{3.5}\label{fakeinq4.4}
\]

\end{lemma}
\begin{proof}
For any $\epsilon\in (0,1)$, set $v_\epsilon:=\max\{v-\epsilon,0 \}$. Lemma \ref{lpschcom} in Appendix \ref{Aapp} implies $v_\epsilon\in D^{1,p}(\Omega,r^{p+\beta})$.

Let $(t,y)$ be the polar coordinate system around $o$. The curvature assumption together with
 (\ref{Riccompar2.8}) and Lemma \ref{compinft} implies
\[
\hat{\sigma}_o(t,y)\leq e^{-\tau(y)}t^{n-1}, \text{ for any }y\in S_oM, \  0<t<i_y.\tag{3.6}\label{3.5volume}
\]

Firstly, we show that  $\|v\|_D$ is finite. In fact, since $pc(\delta)+\beta+n>0$ and $pc(\delta)-(\beta+n)>0$, (\ref{formulavolume}) together with (\ref{3.5volume}) furnishes
\begin{align*}
&\int_M r^\beta |v|^p d\mathfrak{m}
=\int_{B_o(s)}\left( \frac{r}{s} \right)^{pc(\delta)} r^\beta d\mathfrak{m}+\int_{M\backslash B_o(s)}\left(   \frac{r}{s}\right)^{-pc(\delta/2)}r^\beta d\mathfrak{m}\\
\leq &s^{\beta+n}\int_{S_oM}e^{-\tau(y)}d\nu_o(y)\left[ \frac{1}{pc(\delta)+\beta+n}+\frac{1}{pc(\delta/2)-(\beta+n)}       \right]<+\infty.\tag{3.7}\label{sharpness 4.4}
\end{align*}
On the other hand, a direct calculation together with (\ref{sharpness 4.4}) yields
\begin{align*}
&\int_M F^{*p}(dv)r^{p+\beta}d\mathfrak{m}=\int_{B_o(s)}F^{*p}(dv)r^{p+\beta}d\mathfrak{m}+\int_{M\backslash B_o(s)}F^{*p}(dv)r^{p+\beta}d\mathfrak{m}\\
=& c^p(\delta)\int_{B_o(s)}r^\beta |v|^p d\mathfrak{m}+c^p(\delta/2)\int_{M\backslash B_o(s)}r^\beta |v|^p d\mathfrak{m}<+\infty.\tag{3.8}\label{sharpness 4.5}
\end{align*}
Then the finiteness of $\|v\|_D$ follows. Moreover, since $c^p(\delta)>c^p(\delta/2)>0$, (\ref{fakeinq4.4}) follows  from (\ref{sharpness 4.5})  immediately.

Secondly, we prove $\|v_\epsilon-v\|_D\rightarrow$ as $\epsilon\rightarrow 0^+$. Choose a small $\epsilon\in (0,1)$ such that $\mathfrak{i}_o>s\epsilon^{1/c(\delta)}$.
Due to the positivity of  $r^{\beta+n-1}$,  (\ref{3.5volume}) yields
\begin{align*}
&\int_{B_o(s \epsilon^{-1/c(\delta/2)})\backslash B_o(s\epsilon^{1/c(\delta)})}\epsilon^pr^\beta d\mathfrak{m}
\leq\epsilon^p\int_{S_oM}e^{-\tau(y)}d\nu_o(y)\int^{\min\{i_y,s \epsilon^{-1/c(\delta/2)}\}}_{s\epsilon^{1/c(\delta)}}r^{\beta+n-1}dr\\
\leq&\frac{s^{\beta+n}}{\beta+n}\left[ \epsilon^{\frac{pc(\delta/2)-(\beta+n)}{c(\delta/2)}} -\epsilon^{\frac{pc(\delta)+\beta+n}{c(\delta)}} \right]\int_{S_oM}e^{-\tau(y)}d\nu_o(y)\rightarrow 0,\text{ as }\epsilon\rightarrow0^+,
\end{align*}
which together with $\|v\|_D<\infty$ implies
\begin{align*}
\|v_\epsilon-v\|_D^p
=&\int_{B_o(s\epsilon^{1/c(\delta)})}+\int_{  M\backslash B_o(s \epsilon^{-1/c(\delta/2)})}|v|^pr^\beta d\mathfrak{m}+\int_{B_o(s \epsilon^{-1/c(\delta/2)})\backslash B_o(s\epsilon^{1/c(\delta)})}\epsilon^pr^\beta d\mathfrak{m}\\
&+\int_{B_o(s\epsilon^{1/c(\delta)})}+\int_{  M\backslash B_o(s \epsilon^{-1/c(\delta/2)})}F^{*p}(dv)r^{p+\beta} d\mathfrak{m}\rightarrow0, \text{ as }\epsilon\rightarrow0^+.
\end{align*}
Now $v\in D^{1,p}(\Omega,r^{\beta+p})$ follows from $v_\epsilon\in D^{1,p}(\Omega,r^{p+\beta})$.
\end{proof}

\subsection{Finsler manifolds with non-positive flag curvature}
In this subsection, we study the Hardy inequalities on FMMMs with non-positive flag curvature. To begin with, we review the Laplacian comparison theorems concerned with the flag curvature.
\begin{lemma}[cf.\,\cite{Sh1,WX}]\label{flagcurLa}
Let $(M,o,F,d\mathfrak{m})$ be an $n$-dimensional forward complete PFMMM with $\mathbf{K}\leq 0$. Then the following inequalities hold a.e. on $M$:
\begin{align*}\left\{
\begin{array}{lll}
\Delta r_+\geq \frac{n-1}{r_+}, &\text{ if }\mathbf{S}_o^+\leq 0,\\
\tag{3.9}\label{6.1}\\
-\Delta(-r_-)\geq \frac{n-1}{r_-}, &\text{ if }\mathbf{S}_o^-\geq 0.
\end{array}
\right.
\end{align*}
\end{lemma}
\begin{proof}[Sketch of the proof] A standard argument (see Shen\cite{Sh1} or Wu and Xin \cite{WX}) furnishes $\Delta r_+\geq \frac{n-1}{r_+}$.
 In order to show $-\Delta(-r_-)\geq \frac{n-1}{r_-}$,   consider the   RFMMM $(M,\overleftarrow{F},d\mathfrak{m})$.  Thus, (\ref{revesequn}) together with Remark \ref{Scurvature} yields $\overleftarrow{\mathbf{K}}\leq0,\overleftarrow{\mathbf{S}}^+_o\leq 0$ and therefore, the same argument implies $\overleftarrow{\Delta}(\overleftarrow{r_+})\geq\frac{n-1}{\overleftarrow{r_+}}$. We conclude the proof by $\overleftarrow{\Delta}(\overleftarrow{r_+})=-\Delta(-r_-)$ and $\overleftarrow{r_+}=r_-$.
\end{proof}

\begin{theorem}\label{nonpositflag22}Let $(M,o,F,d\mathfrak{m})$ be an $n$-dimensional forward complete PFMMM with $\mathbf{K}\leq 0$ and $\mathbf{S}_o^-\geq 0$.  Let $\Omega$ be a natural domain with $o\in \Omega$. Given $p,\beta \in \mathbb{R}$ with $p\in (1,n)$ and $\beta>-n$,
 we have
\begin{align*}
\int_\Omega r_-^{\beta+p}\max\{F^{*p}(\pm du)\}d \mathfrak{m}\geq \left(\frac{n+\beta}{p}\right)^p\int_\Omega r_-^\beta |u|^p d\mathfrak{m},\ \forall\,u\in C^\infty_0(\Omega),\tag{3.10}\label{5.1.2againagain}
\end{align*}
 In particular, the constant $ \left( \frac{n+\beta}{p} \right)^p$  is sharp if $\lambda_F(\Omega)=1$.
\end{theorem}

\begin{proof}Let $\alpha=(n-p)/(p-1)$ and  $c=\alpha[(\alpha+1)(p-1)+\beta+1]>0$. A direct calculation together with (\ref{6.1}) yields
\begin{align*}
-c\Delta_p(r^{-\alpha}_-)=c\alpha^{p-1} r_-^{-(\alpha+1)(p-1)-1}\left[ -(\alpha+1)(p-1)+r_-(-\Delta(-r_-)) \right]\geq 0.
\end{align*}
And Lemma \ref{centerinteg} yields  $r_-^{(-\alpha-1)(p-1)},r_-^\beta,r_-^{\beta+p}\in L^1_{\lo}(\Omega)$. Thus, Lemma \ref{mainlemmforcr} (ii) furnishes (\ref{5.1.2againagain}) immediately.

In the sequel, we show that the constant $\vartheta^p:=\left( ({n+\beta})/{p} \right)^p$ is sharp if $F|_{\Omega}$ is reversible. Set
\[
\mathfrak{C}_{\beta,p}(\Omega):=\inf_{u\in C^\infty_0(\Omega)\backslash\{0\}}\frac{\int_{\Omega} r^{\beta+p}F^{*p}( du)d \mathfrak{m}}{\int_{\Omega} r^\beta|u|^p d\mathfrak{m}}.
\]
Thus, (\ref{5.1.2againagain}) furnishes $\mathfrak{C}_{\beta,p}(\Omega)\geq \vartheta^p$. On the other hand, choose $R\in (0,\mathfrak{i}_o)$ such that ${B_o(R)}\subset\subset \Omega$.
For any $\epsilon\in (0,R/4)$,
define $u_\epsilon(x):=\max\{\epsilon,r(x) \}^{-\vartheta}$
and choose a cut-off function $\phi\in C^\infty_0(\Omega)$ such that
\begin{align*}\phi=\left\{
\begin{array}{lll}
1, && x\in B_o(R/2),\\
\\
0, && x\notin B_o(R).
\end{array}
\right.
\end{align*}
Set $v:=\phi u_\epsilon$. Lemma \ref{lpschcom} implies $v\in D^{1,p}(\Omega,r^{\beta+p})$.
And a direct calculation yields
\begin{align*}
\int_{\Omega} r^{{\beta+p}}F^{*p}( dv) d\mathfrak{m}=\vartheta^p\int_{B_o(R/2)\setminus B_o(\epsilon)}r^{-n}d\mathfrak{m}+\int_{B_o(R)\setminus  B_o(R/2)} r^{{\beta+p}}F^{*p}( d(\phi r^{-\vartheta}) )d\mathfrak{m}.
\end{align*}
On the other hand, Lemma \ref{centerinteg} implies
\begin{align*}
\int_{\Omega} r^{\beta}|v|^p d\mathfrak{m}\geq \int_{B_o(R/2)\setminus B_o(\epsilon)}r^{\beta}|v|^p d\mathfrak{m}=\int_{B_o(R/2)\setminus B_o(\epsilon)}r^{-n}d\mathfrak{m}\rightarrow +\infty,\ \text{ as }\epsilon\rightarrow0^+.
\end{align*}
Therefore, the above inequalities furnish
\begin{align*}
\mathfrak{C}_{\beta,p}(\Omega)\leq&\frac{\int_{\Omega} r^{{\beta+p}}F^{*p}( dv)d\mathfrak{m}}{\int_{\Omega} r^{{\beta}}|v|^p d\mathfrak{m}}\rightarrow \vartheta^p, \text{ as }\epsilon\rightarrow0^+,
\end{align*}
which concludes the proof.
\end{proof}

\begin{remark}\label{strongconditH}
Let $(M,F,d\mathfrak{m})$ be an $n$-dimensional forward complete FMMM with $\mathbf{K}\leq 0$ and $\mathbf{S}\geq 0$ and let $\Omega$ be a natural domain. The same argument as above shows that (\ref{5.1.2againagain}) remains valid even if $o\notin \Omega$.
\end{remark}

\begin{proof}[Proof of Theorem \ref{nonpositflag2}] Theorem \ref{nonpositflag2} is a direct consequence of Theorem \ref{nonpositflag22}.
\end{proof}

Recently, the $L^p$-Hardy inequalities on a reversible Minkowski space endowed with the Lebesgue measure have been investigated in Mercaldo, Sano and Takahshi \cite{MST} by different methods.
The following example follows from \cite[Theorem 1.1, (2.2), Theorem 6.4]{MST}, which can also be  deduced from Theorem \ref{nonpositflag22} and Remark \ref{strongconditH}.
\begin{example}
Let $(\mathbb{R}^n, F, dx)$ be an $n$-dimensional reversible Minkowski space endowed with the Lebesgue measure. Let $\Omega$ be a  domain in $\mathbb{R}^n$ and let $r(x):=d_F(\mathbf{0},x)=F(x)$. For $1<p<n$, one has
\[
\left(\frac{n-p}{p}\right)^p\int_\Omega \frac{|u(x)|^p}{r^p(x)}dx\leq \int_\Omega {F^p(\nabla u)}dx, \ \forall\,u\in C^\infty_0(\Omega).
\]
Moreover, if $\mathbf{0}\in \Omega$, then $\left(\frac{n-p}{p}\right)^p$ is sharp (but not attained).
\end{example}

\begin{theorem}\label{frirstcurHard2}
Let $(M,o,F,d\mathfrak{m})$ be an $n$-dimensional forward complete PFMMM with $\mathbf{K}\leq 0$ and $\mathbf{S}_o^+\leq 0$.  Given any $p,\beta\in \mathbb{R}$ with $p>n>-\beta$,
we have
\[
\int_{M}r^{\beta+p}_+\max\{ F^{*p}(\pm du) \}d\mathfrak{m}\geq \left( \frac{n+\beta}{p} \right)^p\int_{M}r_+^{\beta} {|u|^p}d\mathfrak{m},\ \forall\,u\in C^\infty_0(M).\tag{3.11}\label{Hard5.3}
\]
In particular, the constant $ \left( \frac{n+\beta}{p} \right)^p$  is sharp if  $\lambda_F=1$, $\mathfrak{i}_o=+\infty$, $\mathbf{K}=0$ and  $\mathbf{S}_o^+=0$.
\end{theorem}

\begin{proof}
 Let $\alpha:=(p-n)/(p-1)$ and $c:=\alpha[(\alpha-1)(p-1)-\beta-1]<0$. A direct calculation together with  (\ref{6.1}) yields
\begin{align*}
-c\Delta_p(r_+^\alpha)=-c\alpha^{p-1}r_+^{(\alpha-1)(p-1)-1}\left[  (\alpha-1)(p-1)+r_+\Delta r_+ \right]\geq 0.
\end{align*}
And Lemma \ref{centerinteg} implies $r_+^{(\alpha-1)(p-1)},r^{\beta},r^{\beta+p}\in L^1_{\lo}({M})$. Then
 (\ref{Hard5.3}) follows from  Lemma \ref{mainlemmforcr} (i) directly.

It remains to show that $\vartheta^p:=\left( ({n+\beta})/{p} \right)^p$ is sharp if
$\lambda_F=1$, $\mathfrak{i}_o=+\infty$, $\mathbf{K}=0$ and $\mathbf{S}_o^+=0$. Note that $M$ is noncompact in this case.
Now set
\[
\mathfrak{C}_{\beta,p}({M}):=\inf_{u\in C^\infty_0({M})\backslash\{0\}}\frac{\int_{M} r^{\beta+p}F^{*p}( du)d \mathfrak{m}}{\int_{M} r^\beta|u|^p d\mathfrak{m}}.%\tag{4.6}\label{best4.6}
\]
Thus, (\ref{Hard5.3}) implies $\mathfrak{C}_{\beta,p}({M})\geq \vartheta^p$. On the other hand, due to $\mathbf{Ric}=0$, Lemma \ref{impsharplem} furnishes
\[
\mathfrak{C}_{\beta,p}({M})\leq \frac{\int_{M} r^{\beta+p}F^{*p}( dv)d \mathfrak{m}}{\int_{M} r^\beta|v|^p d\mathfrak{m}}<c^p(\delta)\rightarrow \vartheta^p,\text{ as }\delta\rightarrow 0^+,
\]
where $v$ is defined as in Lemma \ref{impsharplem}. This concludes the proof.
\end{proof}

\begin{proof}[Proof of Theorem \ref{frirstcurHard}] Since the injectivity radius of every point in a Cartan-Hadamard manifold is infinite (cf. Bao et al. \cite{BCS}), Theorem \ref{frirstcurHard} follows from Theorem \ref{frirstcurHard2} immediately.
\end{proof}

A reversible Minkowski space is a linear space equipped with a reversible Minkowski norm.
According to Bao et al. \cite{BCS} and Shen \cite{Shen2013,Sh1},
every reversible Minkowski space endowed with the Lebesgue measure is a Cartan-Hadamard measure manifold with $\lambda_F=1$, $\mathbf{K}=0$ and  $\mathbf{S}=0$.
Hence, the inequality (\ref{Hard5.3}) is optimal on such spaces.
There are various reversible Minkowski norms on $\mathbb{R}^n$. Now we recall two types of them.
\begin{example}\label{ConstrucMin}
1. $(\alpha,\beta)$-metrics on $\mathbb{R}^n$. See Chern and Shen \cite{CHZ}. Let $\phi(s)$ be  an arbitrary $C^\infty$ even function on some symmetric
open interval $I = (-b_0, b_0)$ with
\[
\phi(s)>0,\ (\phi(s)-s\phi'(s))+(b^2-s^2)\phi''(s)>0,
\]
where $s$ and $b$ are arbitrary numbers with $|s| < b < b_0$. Then for any constant $1$-form $\beta(y)=b_iy^i$ on $\mathbb{R}^n$ with $\|\beta\|<b_0$, the following metric
\[
F(y):=\|y\|\, \phi\left(  \frac{\beta(y)}{\|y\|} \right), \ \forall \, y\in T\mathbb{R}^n,
\]
is a reversible Minkowski norm on $\mathbb{R}^n$, where $\|\cdot\|$ denotes the  Euclidean norm.

\smallskip

2. Fourth Root metrics on $\mathbb{R}^n$. See Li and Shen \cite{LS}. Let $A(y):=a_{ijkl}y^iy^jy^ky^l$. Suppose  $(2AA_{ij}-A_iA_j)$  is positive definite, where $A_i:=\frac{\partial A}{\partial y^i}$ and $A_{ij}:=\frac{\partial^2 A}{\partial y^i\partial y^j}$. Then $F(y):=A^{\frac14}$ is a reversible Minkowski norm on $\mathbb{R}^n$.
\end{example}

Given a forward complete Finsler manifold $(M,F)$,  the topology induced by  backward balls is the same as the original one (cf. Bao et al. \cite[p.155]{BCS}).
Denote by $\overleftarrow{\mathfrak{i}_o}$ (resp., $\mathfrak{i}_o$ ) the injectivity radius of $o$ with respect to $(M,\overleftarrow{F})$ (resp., $(M,F)$).  A standard argument shows
both $\overleftarrow{\mathfrak{i}_o},\mathfrak{i}_o$ are positive (see Bao et al. \cite[Theorem 6.3.1]{BCS} and Yuan et al. \cite[Proposition 3.2]{YZS} for example). However, $\overleftarrow{\mathfrak{i}_o}$ may not coincide with $\mathfrak{i}_o$ if $\lambda_F\neq1$.
For instance, for the   Funk manifold in Example \ref{Funkexam}, $\mathfrak{i}_{\mathbf{0}}=+\infty$ while $\overleftarrow{\mathfrak{i}_{\mathbf{0}}}=\log 2$. Now we have the following result.

\begin{theorem}\label{logone2}
Let $(M,o,F,d\mathfrak{m})$ be an $n$-dimensional forward complete PFMMM with $\mathbf{K}\leq0$ and $\mathbf{S}_o^-\geq0$.
For any $R,p,\beta\in \mathbb{R}$ with $R\in (0,\overleftarrow{\mathfrak{i}_o}),\,
p\in (1,n]$ and $\beta<-1,$
 we have
\begin{align*}
\int_{B_o^-(R)}\left[\log\left(\frac{R}{r_-} \right) \right]^{p+\beta}\max\{ F^{*p}(\pm du) \}d\mathfrak{m}\geq \left(\frac{|\beta+1|}{p}\right)^p\int_{B_o^-(R)}\left[\log\left(\frac{R}{r_-} \right) \right]^\beta \frac{|u|^p}{r_-^p}d\mathfrak{m},\tag{3.12}\label{5.8lognew}
\end{align*}
for any $u\in C^\infty_0(B_o^-(R))$.
In particular, the constant $\left(\frac{|\beta+1|}{p}\right)^p$ is sharp if $\lambda_F({B_o^-(R)})=1$.
\end{theorem}

\begin{proof} The proof is divided into three steps.

\noindent\textbf{Step 1.} Let $\rho:=\log\left(\frac{R}{r_-} \right)$. Obviously, $\rho^\beta$ can be viewed as a  continuous function on ${B_o^-(R)}$ by setting $\rho^\beta(o)=0$. Now we claim
\[
F^{(p-1)}(\nabla\rho),\rho^\beta F^p(\nabla\rho),\rho^{\beta+p},\rho^{\beta+1}F^{p-1}(\nabla\rho)\in L^1_{\lo}({B_o^-(R)}).\tag{3.13}\label{5.9log}
\]

\noindent (1) Since $p\leq n$, Lemma \ref{centerinteg} implies
$F^{(p-1)}(\nabla\rho)=1/{r_-^{p-1}}\in L_{\lo}^1({B_o^-(R)})$.

\noindent(2)
Let $(\mathfrak{t},\mathfrak{y})$ be the polar coordinate system around $o$ in the RFMMM $(M,\overleftarrow{F},d\mathfrak{m})$. Thus, $\mathfrak{t}=r_-(x)$.
  According to (\ref{v14-2.1}), there exists an $\epsilon\in (0,\min\{1,{R}/2\})$  such that for any $\yyy\in \overleftarrow{S_oM}$,
\[
0<\overleftarrow{\hat{\sigma}_o}(\ttt,\yyy)<2 e^{-\overleftarrow{\tau}(\yyy)}{\ttt}^{n-1},\ 0<\ttt<\epsilon.\tag{3.14}\label{imp4.12}
\]
Here, we use $\overleftarrow{*}$ to denote the corresponding geometric quantity $*$ in $(M,\overleftarrow{F})$.
Since $n-1-p\geq -1$ and $\beta<-1$,  (\ref{formulavolume})  together with (\ref{revesequn}) and (\ref{imp4.12}) yields
\begin{align*}
\int_{{B^-_o(\epsilon)}}\rho^\beta F^p(\nabla\rho) d\mathfrak{m}=&\int_{\overleftarrow{B^+_o(\epsilon)}}\rho^\beta F^p(\nabla\rho) d\mathfrak{m}=\int_{\overleftarrow{S_oM}}d\overleftarrow{\nu_o}(\yyy)\int_0^\epsilon\left[\log\left( \frac{R}{\ttt}\right)  \right]^\beta\frac{1}{{\ttt}^{p}}\overleftarrow{\hat{\sigma}_o}(\ttt,\yyy)d\ttt\\
\leq &2\int_{\overleftarrow{S_oM}}e^{-\overleftarrow{\tau}(\yyy)}d\overleftarrow{\nu_o}(\yyy)\int_0^\epsilon \left[ \log\left( \frac{R}{\ttt} \right) \right]^\beta {\ttt}^{n-1-p}d\ttt\\
\leq&\frac{2}{|1+\beta|}\left[\log\left( \frac{R}{\epsilon}\right)\right]^{\beta+1}\int_{\overleftarrow{S_oM}}e^{-\overleftarrow{\tau}(\yyy)}d\overleftarrow{\nu_o}(\yyy)<+\infty,\tag{3.15}\label{newes4.13}
\end{align*}
 which implies  $\rho^\beta F^p(\nabla\rho)\in L^{1}_{\lo}({B_o^-(R)})$.

\smallskip

\noindent(3) Since $\beta+p<n-1$, we choose a small $\epsilon\in (0,\min\{1,{R}/2\})$ such that
\[
\frac{(\log s)^{\beta+p}}{s^{n-1}}\leq 1, \text{ for }s\in(R/\epsilon,+\infty).
\]
A calculation similar to (\ref{newes4.13}) then yields
\begin{align*}
\int_{{B^-_o(\epsilon)}}\rho^{\beta+p}d\mathfrak{m}
\leq & 2\int_{\overleftarrow{S_oM}}e^{-\overleftarrow{\tau}(\yyy)}d\overleftarrow{\nu_o}(\yyy)\int_0^\epsilon\left[\log\left( \frac{R}{\ttt}\right)  \right]^{\beta+p}{\ttt}^{n-1}d\ttt\\
\leq& 2{R}^{n-1}\epsilon\int_{\overleftarrow{S_oM}}e^{-\overleftarrow{\tau}(\yyy)}d\overleftarrow{\nu_o}(\yyy)<+\infty,%\tag{3.15}\label{newes4.14}
\end{align*}
which   implies $\rho^{\beta+p}\in L^1_{\lo}({B_o^-(R)})$.

\smallskip

\noindent(4) The H\"older inequality together with (2) and (3)  furnishes $\rho^{\beta+1}F^{p-1}(\nabla\rho)\in L^1_{\lo}({B_o^-(R)})$.

\smallskip

Therefore, (\ref{5.9log}) follows as claimed.

\smallskip

\noindent\textbf{Step 2.}
In this step, we prove (\ref{5.8lognew}). In order to do this, set
$X:=-\rho^{\beta+1}F^{p-2}(\nabla\rho)\nabla\rho$ and $f_X:=c\rho^\beta F^{p}(\nabla\rho)$,
where  $c:=-\beta-1=|\beta+1|>0$. Then (\ref{5.9log}) indicates $X\in L^1_{\lo}(TB_o^-(R))$, $f_X\in L_{\lo}^1({B_o^-(R)})$ and $ F^p(X)/f^{p-1}_X\in L^1_{\lo}({B_o^-(R)})$.

If $f_X\leq \di X$ in the weak sense, then (\ref{5.8lognew}) would follow  from Theorem \ref{divlemf} (1)  immediately. Therefore, it suffices to show that $f_X\leq \di X$ in the weak sense, i.e.,
\begin{align*}
\int_{B_o^-(R)} \rho^{\beta+1}F^{p-2}(\nabla\rho)\langle  \nabla\rho,du\rangle d\mathfrak{m}\geq c\int_{B_o^-(R)}\rho^\beta F^p(\nabla\rho)ud\mathfrak{m},\ \forall\,u\in C^\infty_0({B_o^-(R)}) \text{ with }u\geq 0.\tag{3.16}\label{5.11lognewww}
\end{align*}

We proceed as follows.
Set $v:=\rho^{-c} u$. Thus, (\ref{5.9log}) together with (\ref{2.2newineq}) yields
\begin{align*}
&\langle  F^{p-2}(\nabla\rho)\nabla\rho,u\rho^{-c-1}d\rho \rangle= F^p(\nabla\rho) \rho^\beta u\in L^1(B_o^-(R)),\\
&\left|\langle  F^{p-2}(\nabla\rho)\nabla\rho,\rho^{-c}d u \rangle\right|\leq\lambda_F\left(\text{supp}(u)\right)\, \rho^{\beta+1}F^{p-1}(\nabla\rho)F^*(du)\in L^1(B_o^-(R)),
\end{align*}
which imply
\[
\langle F^{p-2}(\nabla\rho)\nabla\rho,dv\rangle=-c\langle  F^{p-2}(\nabla\rho)\nabla\rho,u\rho^{-c-1}d\rho \rangle+\langle  F^{p-2}(\nabla\rho)\nabla\rho,\rho^{-c}d u \rangle\in L^1(B_o^-(R)).\tag{3.17}\label{inter4.18}
\]
Choose a small $\delta\in (0,\epsilon)$, where $\epsilon$ is defined as in (\ref{imp4.12}). Thus, (\ref{inter4.18}) furnishes
 \[
 \lim_{\delta\rightarrow 0^+}\int_{\overline{B^-_o(\delta)}}\langle F^{p-2}(\nabla\rho)\nabla\rho,dv\rangle d\mathfrak{m}=0.\tag{3.18}\label{lim4.19}
 \]
 On the other hand,  (\ref{6.1}) yields
$
-\Delta_p \rho=r^{-p}_-\left[(1-p)+r_-(-\Delta(-r_-))  \right]\geq \frac{(n-p)}{r^p_-}\geq0$,
which together with  (\ref{olddivlem}) yields
\begin{align*}
0\leq &-\int_{B_o^-(R)\backslash\overline{{B^-_o(\delta)}}}v\Delta_p\rho d\mathfrak{m} =-\int_{B_o^-(R)\backslash\overline{{B^-_o(\delta)}}}\left[\di\left(v F^{p-2}(\nabla\rho)\nabla\rho   \right)-\langle F^{p-2}(\nabla\rho)\nabla\rho, dv   \rangle    \right]d\mathfrak{m}\\
=&-\int_{\partial{B^-_o(\delta)}}g_{\mathbf{n}_-}(\mathbf{n}_-,F^{p-2}(\nabla\rho)\nabla\rho)v dA+\int_{B_o^-(R)\backslash\overline{{B^-_o(\delta)}}}\langle F^{p-2}(\nabla\rho)\nabla\rho, dv   \rangle d\mathfrak{m},\tag{3.19}\label{div4.14}
\end{align*}
where  $\mathbf{n}_-$ is  the unit inward normal vector field on $\partial{B^-_o(\delta)}$ and $dA$ is the induced measure on $\partial{B^-_o(\delta)}$.

Let $(\ttt,\yyy)$ be the polar coordinates as in Step 1. Thus, $dA=\overleftarrow{\hat{\sigma}_o}(\delta,\yyy)d\overleftarrow{\nu_o}(\yyy)$.
Since $\delta\in (0,\epsilon)$, (\ref{pre2.1}) together with (\ref{imp4.12}) yields
\begin{align*}
&\left| \int_{\partial{B^-_o(\delta)}}g_{\mathbf{n}_-}(\mathbf{n}_-,F^{p-2}(\nabla\rho)\nabla\rho)v dA \right|\leq \lambda_F(\partial{B^-_o(\delta)})\int_{\partial{B^-_o(\delta)}}F^{p-1}(\nabla\rho)\rho^{-c} u dA\\
\leq &2\lambda_F(\partial{B^-_o(\delta)})\delta^{n-p}\left[ \log\left( \frac{R}{\delta} \right) \right]^{-c}\max_{\partial{B^-_o(\delta)}}u\,\int_{\overleftarrow{S_oM}}e^{-\overleftarrow{\tau}(\yyy)}d\overleftarrow{\nu_o}(\yyy)\rightarrow 0, \text{ as }\delta\rightarrow 0^+.\tag{3.20}\label{4.16divlem}
\end{align*}
Now it follows from  (\ref{lim4.19})-(\ref{4.16divlem}) that
$\int_{B_o^-(R)}\langle F^{p-2}(\nabla\rho)\nabla\rho, dv   \rangle d\mathfrak{m}\geq 0$,
which together with (\ref{inter4.18}) furnishes (\ref{5.11lognewww}).

\smallskip

\noindent\textbf{Step 3.} Now we show the constant $\vartheta^p:=\left( \frac{|\beta+1|}{p} \right)^p$ is sharp if $\lambda_F({{B^-_o(R)}})=1$. In this case, $r_+(x)=r_-(x)=:r(x)$ for all $x\in B^-_o(R)$ and hence,
$B^+_o(R)=B^-_o(R)=:B_o(R)$ and $\mathfrak{i}_o\geq R$.
Set
\[
\mathfrak{C}_{\beta,p}(B_o(R)):=\inf_{u\in C^\infty_0({B_o(R)})\setminus\{0\}}\frac{\int_{B_o(R)}\rho^{p+\beta}F^{*p}( du)d\mathfrak{m}}{\int_{B_o(R)}\rho^\beta \frac{|u|^p}{r^p}d\mathfrak{m}}.
\]
Thus, (\ref{5.8lognew}) implies $\mathfrak{C}_{\beta,p}(B_o(R))\geq \vartheta^p$. For the reverse inequality,
 set $c(\delta):=\frac{|\beta+1|+\delta}{p}$ for $\delta\in (0,1)$ and
define a function $v$ on $B_o(R)$ by
\begin{align*}v(x):=\left\{
\begin{array}{lll}
\left(\frac{\rho(x)}{s}\right)^{-c(\delta/2)}, && \text{ if  }x\in B_o(R/2),\\
\\
\left(\frac{\rho(x)}{s}\right)^{c(\delta)}, && \text{ if  }x\in {B_o(R)}\backslash B_o(R/2),
\end{array}
\right.
\end{align*}
where $s:=\log 2>0$.

By (\ref{5.8lognew}), we  define $D^{1,p}(B_o(R),\rho^{\beta+p})$ as  the completion of $C^\infty_0(B_o(R))$ with respect to the norm $|u|_D:=\left(\int_{B_o(R)}\rho^{p+\beta}F^{*p}( du)d\mathfrak{m}\right)^{\frac1p}$.
If $v\in D^{1,p}(B_o(R),\rho^{\beta+p})$, then an easy calculation similar to (\ref{sharpness 4.5}) would furnish
\[
\mathfrak{C}_{\beta,p}(B_o(R))\leq \frac{\int_{B_o(R)}\rho^{p+\beta}F^{*p}( dv)d\mathfrak{m}}{\int_{B_o(R)}\rho^\beta \frac{|u|^p}{r^p}d\mathfrak{m}}<c^p(\delta)\rightarrow \vartheta^p, \text{ as }\delta\rightarrow 0^+,
\]
which would conclude the proof. Hence, it suffices to show $v\in D^{1,p}(B_o(R),\rho^{\beta+p})$.

In order to do this, we first prove $|v|_D<\infty$. Let $(t,y)$ be the polar coordinate system around $o$ in $(M,F)$.
Since $\mathbf{K}\leq 0$ and $\overline{{B_o(R)}}$ is compact (cf. Bao et al. \cite[Theorem 6.6.1]{BCS}),  Remark \ref{Scurvature} together with (\ref{revesequn}) and $\lambda_F(B_o(R))=1$ yields two constants $k,h\in \mathbb{R}$ with
\[
-k^2\leq\mathbf{Ric}(\nabla t)\leq 0,\ -h^2\leq \mathbf{S}(\nabla t)\leq 0,\ \forall \,y\in S_oM,\ 0\leq t\leq R,
\]
which together with (\ref{Riccompar2.8}) yields
\[
\hat{\sigma}_o(t,y)\leq e^{-\tau(y)+h^2t} \mathfrak{s}^{n-1}_{-k^2}(t), \ \forall \,y\in S_oM, \ 0<t<R\leq \mathfrak{i}_o.\tag{3.21}\label{volcomp4.22}
\]
Now we study $\int_{B_o(R)}\rho^{p+\beta}F^{*p}( dv)d\mathfrak{m}$.  Let
$\epsilon$ be defined as in (\ref{imp4.12}).
  Since $\beta+1<0$ and $n-1-p\geq -1$, (\ref{formulavolume}) together with (\ref{imp4.12}) then yields
\begin{align*}
\int_{B_o(\epsilon)}\rho^{p+\beta}F^{*p}( dv)d\mathfrak{m}
\leq \frac{2 c^p(\delta/2)}{s^{-c(\delta/2)p}}\int_{S_oM}e^{-\tau(y)}d\nu_o(y)\int^\epsilon_0\left[\log\left( \frac{R}{t} \right)\right]^{\beta-c(\delta/2)p}\frac1tdt<+\infty.\tag{3.22}\label{finitepart1}
\end{align*}
On the other hand, (\ref{formulavolume})  together with (\ref{volcomp4.22}) implies
\begin{align*}
\int_{{B_o(R)}\backslash B_o({R/2})} \rho^{p+\beta}F^{*p}( dv) d\mathfrak{m}
\leq&\left( \frac{2 c(\delta)}{Rs^{c(\delta)}} \right)^p\int_{S_oM}e^{-\tau(y)}d\nu_o(y)\int_{R/2}^R \left[\log\left( \frac{R}{t}\right)\right]^{-1+\delta} { \mathfrak{s}^{n-1}_{-k^2}(t)e^{h^2t}}dt\\
\leq& \frac{(2c(\delta))^p e^{h^2R}\mathfrak{s}^{n-1}_{-k^2}(R)}{R^{p-1}s^{c(\delta p)}}\int_{S_oM}e^{-\tau(y)}d\nu_o(y)\int_{R/2}^R \left[\log\left( \frac{R}{t}\right)\right]^{-1+\delta}\frac{dt}{t}<+\infty,
\end{align*}
which together with  (\ref{finitepart1}) indicates
 $\int_{B_o(R)}\rho^{p+\beta}F^{*p}( dv)d\mathfrak{m}<\infty$, i.e.,
 $|v|_D<\infty$.

 Now set $v_\epsilon:=\max\{v-\epsilon,0\}$.
By the finiteness of $|v|_D$, one can easily check $|v_\epsilon-v|_D\rightarrow 0$ as $\epsilon\rightarrow 0^+$. On the other hand, a similar argument as in the proof of Lemma \ref{lpschcom} yields $v_\epsilon\in D^{1,p}(B_o(R),\rho^{\beta+p})$.
Therefore,
 $v\in  D^{1,p}(B_o(R),\rho^{\beta+p})$.
\end{proof}

\begin{proof}[Proof of Theorem \ref{logone}] For a reversible Cartan-Hadamard manifold, $\overleftarrow{\mathfrak{i}_o}=\mathfrak{i}_o=+\infty$. Thus, Theorem \ref{logone} follows from  Theorem \ref{logone2} directly.
\end{proof}

\subsection{Finsler manifolds with nonnegative Ricci curvature}
In this subsection, we consider the Hardy inequalities on FMMMs with nonnegative (weighted) Ricci curvature. We begin by recalling the  Laplacian  comparison theorems concerned with the (weighted) Ricci curvature.

\begin{lemma}[cf.\,\cite{Ot,Sh1,WX,Y}]\label{RicLaplace}
Let $(M,o,F, d\mathfrak{m})$ be an $n$-dimensional forward complete PFMMM.

\smallskip

(i) If $\mathbf{Ric}_N\geq 0$, $N\in [n,+\infty)$, then $\Delta r_+\leq \frac{N-1}{r_+}$ and $-\Delta(-r_-)\leq \frac{N-1}{r_-}$ hold a.e. on $M$.

\smallskip

(ii) Suppose either $\mathbf{Ric}_\infty\geq 0$ or $\mathbf{Ric}\geq 0$. Then the following inequalities hold a.e. on $M$:
\begin{align*}\left\{
\begin{array}{lll}
\Delta r_+\leq \frac{n-1}{r_+},  &\text{ if }\mathbf{S}_o^+\geq 0,\\
\\
-\Delta(-r_-)\leq \frac{n-1}{r_-}, &\text{ if }\mathbf{S}_o^-\leq 0.
\end{array}
\right.
\end{align*}
\end{lemma}
\begin{proof}[Sketch of the proof]  First we consider the case of $r_+$. Then (i) is Ohta and Sturm\cite[Theorem 5.2]{Ot} while (ii) follows from the standard Laplacian comparison theorem (cf. Shen \cite{Sh1} or Wu and Xin \cite{WX}) and Lemma  \ref{compinft} (also see Yin \cite[Theorem A]{Y}).
And an argument based on the RFMMM furnishes the  results in the case of $r_-$.
\end{proof}

Now we show Theorem  \ref{Rictheno}.

\begin{proof}[Proof of Theorem \ref{Rictheno}]
\textbf{(1)} Clearly, $\Omega:=M\backslash\{o\}$ is an natural domain. Let $\alpha=(p-N)/(p-1)$ and $c:=\alpha[(\alpha-1)(p-1)-\beta-1]\geq 0$.
A direct calculation together with  Lemma \ref{RicLaplace} (i) furnishes
\begin{align*}
-c\Delta_p(r^\alpha_+)=-c\alpha^{p-1}r_+^{(\alpha-1)(p-1)-1}\left[  (\alpha-1)(p-1)+r_+\Delta r_+ \right]\geq 0.
\end{align*}
Thus, Lemma \ref{mainlemmforcr} (i) yields (\ref{5.12newnew}) immediately. In the sequel, we show $\vartheta^p:=\left( \frac{|n+\beta|}{p} \right)^p$ is sharp if $\lambda_F=1$, $N=n$ and $p+\beta>-n$. In this case, $\mathbf{Ric}_N\geq 0$ means $\mathbf{Ric}=\mathbf{Ric}_n\geq 0$ and $\mathbf{S}=0$. Then the sharpness follows from the same argument as in the proof of Theorem \ref{frirstcurHard2}.

\smallskip

\textbf{(2)} Let $\alpha=(N-p)/(p-1)$ and  $c=\alpha[(\alpha+1)(p-1)+\beta+1]<0$. A direct calculation together with Lemma \ref{mainlemmforcr} (ii) yields
\begin{align*}
-c\Delta_p(r^{-\alpha}_-)=c\alpha^{p-1} r_-^{-(\alpha+1)(p-1)-1}\left[ -(\alpha+1)(p-1)+r_-(-\Delta(-r_-)) \right]\geq 0.
\end{align*}
The rest proof is the same as \textbf{(1)} and hence, we omit it.
\end{proof}

A similar argument also furnishes the $\mathbf{Ric}_\infty/ \mathbf{Ric}$ version of Theorem \ref{Rictheno}. We omit the proof.
\begin{theorem} \label{Rictheno2}
Let $(M,o,F,d\mathfrak{m})$ be an $n$-dimensional forward complete PFMMM with either $\mathbf{Ric}_\infty\geq0$ or $\mathbf{Ric}\geq 0$.

\smallskip

\noindent(1) Suppose $\mathbf{S}_o^+\geq 0$. Given $p,\beta\in \mathbb{R}$ with $p>n$ and $\beta<-n$, for any $u\in C^\infty_0(M\setminus\{o\})$, we have
 \[
\int_{M\backslash\{o\}}r^{\beta+p}_+\max\{ F^{*p}(\pm du) \}d\mathfrak{m}\geq \left( \frac{|n+\beta|}{p} \right)^p\int_{M\backslash\{o\}}r_+^{\beta} {|u|^p}d\mathfrak{m}.
\]

\smallskip

\noindent(2) Suppose $\mathbf{S}_o^-\leq 0$. Given $p,\beta\in \mathbb{R}$ with $1<p<n$ and  $\beta<-n$, for any $u\in C^\infty_0(M\setminus\{o\})$, we have
 \[
\int_{M\backslash\{o\}}r^{\beta+p}_-\max\{ F^{*p}(\pm du) \}d\mathfrak{m}\geq \left( \frac{|n+\beta|}{p} \right)^p\int_{M\backslash\{o\}}r_-^{\beta} {|u|^p}d\mathfrak{m}.
\]

\smallskip

\noindent In particular, the constants in (1) and (2) are sharp  if $\lambda_F=1$ and $p+\beta>-n$.
\end{theorem}

In the sequel, we present two applications of Theorem \ref{Rictheno2}.
\begin{definition}\label{Hanshukongj}
Let $(M,F, d\mathfrak{m})$ be an $n$-dimensional closed reversible FMMM. Given $o\in M$, set $r(x):=d_F(o,x)$. Let $\Omega$ be either $M$ or $M\backslash\{o\}$.
Given $p\in (1,+\infty)$ and $\beta<-n$ with $p+\beta>-n$, define $W^{1,p}(\Omega, r^{\beta+p})$ as the completion of $C^\infty_0(\Omega)$ under the norm
\[
\|u\|_{p,\beta}:=\left( \int_{\Omega}|u|^p{r}^{p+\beta} d\mathfrak{m}+\int_{\Omega} F^{*p}(d u) {r}^{p+\beta} d\mathfrak{m}\right)^{\frac1p}.
\]
\end{definition}

By the compactness of $M$, one can easily verify
 $D^{1,p}(M\backslash\{o\},r^{p+\beta})\subset W^{1,p}(M\backslash\{o\},r^{p+\beta})$, where $D^{1,p}(M\backslash\{o\},r^{p+\beta})$ is defined in Definition \ref{DefDS}.
 Moreover, Theorem \ref{Rictheno2} yields the following result.
\begin{theorem}\label{spaceequ}
Let $(M,o,F,d\mathfrak{m})$ be an $n$-dimensional closed reversible PFMMM with $\mathbf{S}_o^+\geq 0$ and either $\mathbf{Ric}_\infty\geq0$ or $\mathbf{Ric}\geq 0$.
Then for any $p\in (1,n)\cup (n,+\infty)$ and $\beta<-n$ with $p+\beta>-n$, we have
\[
D^{1,p}(M\backslash\{o\},r^{p+\beta})= W^{1,p}(M\backslash\{o\},r^{p+\beta}).
\]
\end{theorem}
\begin{proof}
It is enough to show  $W^{1,p}(M\backslash\{o\},r^{p+\beta})\subset D^{1,p}(M\backslash\{o\},r^{p+\beta})$. For any $u\in W^{1,p}(M\backslash\{o\},r^{p+\beta})$, there exists a sequence $u_j\in C^\infty_0(M\backslash\{o\})$ converging to $u$ under $\|\cdot\|_{p,\beta}$.
Thus, for any $\epsilon>0$, there exists $N>0$ such that for any $j>N$, $\|u_j-u\|_{p,\beta}<\left( \frac{|n+\beta|}{p} \right)\epsilon$. And Lemma \ref{Soleweakder} in Appendix \ref{SObespa} implies that $u_j$ also converges to $u$ pointwise a.e..
Now for $i>N$,
 Fatou's lemma together with Theorem \ref{Rictheno2} yields
\begin{align*}
&\left(\int_{M\backslash\{o\}} |u_i-u|^p r^{\beta}d\mathfrak{m}\right)^{\frac1p}=\left(\int_{M\backslash\{o\}} \underset{j\rightarrow +\infty}{\lim\inf}\,|u_i-u_j|^p r^{\beta}d\mathfrak{m}\right)^{\frac1p}\leq \underset{j\rightarrow +\infty}{\lim\inf}\left(\int_{M\backslash\{o\}} |u_i-u_j|^p r^{\beta}d\mathfrak{m}\right)^{\frac1p}\\
\leq&  \left( \frac{|n+\beta|}{p} \right)^{-1} \underset{j\rightarrow +\infty}{\lim\inf}\left(\int_{M}r^{\beta+p}F^{*p}(du_i-d u_j )d\mathfrak{m}\right)^{\frac1p}\leq  \left( \frac{|n+\beta|}{p} \right)^{-1}\underset{j\rightarrow +\infty}{\lim\inf}\|u_i-u_j\|_{p,\beta}\\
\leq &\left( \frac{|n+\beta|}{p} \right)^{-1}\underset{j\rightarrow +\infty}{\lim\inf}\left(\|u_i-u\|_{p,\beta}+\|u_j-u\|_{p,\beta}  \right)\leq {\epsilon},
\end{align*}
which together with $\|u_i-u\|_{p,\beta}\rightarrow 0$ implies $\|u_i-u\|_D\rightarrow 0$ and hence, $u\in D^{1,p}(M\backslash\{o\},r^{r+\beta})$.
\end{proof}

By the zero extension,  $W^{1,p}(M\backslash\{o\}, r^{\beta+p})$ can be viewed as a subset of $W^{1,p}(M, r^{\beta+p})$. On the other hand, we have the following result, whose proof is postponed until Appendix \ref{SObespa}.
\begin{proposition}\label{thsob} Let $(M,F, d\mathfrak{m})$, $r$, $p$ and $\beta$ be as in Definition \ref{Hanshukongj}.
If $u\in W^{1,p}({M}, {r}^{\beta+p})\cap C({M})$ with $u(o)=0$, then $u|_{M\backslash\{o\}}\in W^{1,p}({{M}\backslash \{o\}}, {r}^{\beta+p})$.
\end{proposition}

The second application of Theorem \ref{Rictheno2} is as follows.

\begin{theorem}\label{reverRicinfty}
Let $(M,o,F,d\mathfrak{m})$ be an $n$-dimensional closed reversible PFMMM with $\mathbf{Ric}_\infty\geq0$ and $\mathbf{S}_o^+\geq 0$. Thus,
 for any $p\in (1,n)\cup(n,\infty)$ and $\beta<-n$ with $p+\beta>-n$, we have
 \[
\int_{M}r^{\beta+p}F^*(d u )^p d\mathfrak{m}\geq \left( \frac{|n+\beta|}{p} \right)^p\int_{M}r^{\beta} {|u|^p}d\mathfrak{m}, \ \forall\, u\in \mathfrak{C}^\infty(M,o),\tag{3.23}\label{new3.26re}
\]
where $\mathfrak{C}^\infty(M,o):=\{u\in C^\infty(M): \,u(o)=0 \}$. In particular, $\left( \frac{|n+\beta|}{p} \right)^p$ is sharp with respect to $\mathfrak{C}^\infty(M,o)$.
\end{theorem}
\begin{proof}
Given $u\in \mathfrak{C}^\infty(M,o)$, Proposition \ref{thsob} together with Theorem \ref{spaceequ} implies that $u|_{M\backslash\{o\}}$ belongs to $D^{1,p}(M\backslash\{o\},r^{p+\beta})$. Hence, there is a sequence $u_j\in C^\infty_0(M\backslash\{o\})$ with $\|u_j-u|_{M\backslash\{o\}}\|_D\rightarrow 0$, which together with Theorem \ref{Rictheno2} furnishes
(\ref{new3.26re}).

For the sharpness of the constant, let $v\in D^{1,p}(M\backslash\{o\},r^{p+\beta})$ be defined as in Lemma \ref{impsharplem}. Thus, there exist  a sequence $v_j\subset C^\infty_0(M\backslash\{o\})$ with $\|v_j- v\|_D\rightarrow0$. By the zero extension, $v_j$ can be viewed as a function in $\mathfrak{C}^\infty(M,o)$. Due to this fact, the rest proof is the same as the one of Theorem \ref{frirstcurHard2}.
\end{proof}

\begin{remark}\label{remarksolesf} Theorem \ref{reverRicinfty} indicates  what kind of function the Hardy inequality remains valid for on a closed manifold. And
Theorem \ref{reverRicinfty}   factually holds for any $u\in W^{1,p}(M,r^{p+\beta})\cap C(M)$ with $u(o)=0$.
\end{remark}

\begin{proof}[Proof of Theorem \ref{cofunRic}.]According to Example \ref{seconex},   the assumption implies $\mathbf{Ric}_\infty\geq 0$ and $\mathbf{S}_o^+\geq 0$.
Thus, Theorem \ref{cofunRic} follows from Theorem \ref{reverRicinfty} directly.
\end{proof}

\section{Hardy inequality for $p$-sub/superharmonic functions}\label{Hardforpharm}

\subsection{A weighted Hardy inequality}

In the section, we study the Hardy inequalities for  $p$-sub/superharmonic functions in the Finsler setting.
Inspired by D'Ambrosio and Dipierro \cite{DD}, we have the following result.

\begin{theorem}\label{wegHardin}Let  $(M,F,d\mathfrak{m})$ be a forward complete FMMM and let $\Omega$ be a natural domain in $M$.
Given $p>1$ and $\alpha\in \mathbb{R}$, let $\rho\in W^{1,p}_{\lo}(\Omega)$ be a nonnegative function satisfying the following conditions:

\smallskip

(1) $-(p-1-\alpha)\Delta_p\rho\geq0$ in the weak sense;

\smallskip

(2) Additionally suppose $\frac{F^p(\nabla \rho)}{\rho^{p-\alpha}}$, $\rho^\alpha\in L^1_{\loc}(\Omega)$ if $\alpha>p$.

\smallskip

\noindent Then we have the following weighted Hardy inequality
\[
\int_\Omega \rho^\alpha \max\{F^{*p}(\pm d{u}) \}d\mathfrak{m}\geq \left( \frac{|p-1-\alpha|}{p} \right)^{p}\int_\Omega \frac{F^p(\nabla\rho)}{\rho^{p-\alpha}}|{u}|^pd\mathfrak{m}, \ \forall \,u\in C^\infty_0(\Omega).\tag{4.1}\label{weightedhard}
\]
\end{theorem}

\begin{proof}
If $p-\alpha-1=0$, (\ref{weightedhard}) is trivial. So we assume $p-\alpha-1\neq0$ in the sequel. Given $\varepsilon\in (0,1)$, set $\rho_\varepsilon:=\rho+\varepsilon$ and
\[
X:=-(p-1-\alpha)\frac{F^{p-2}(\nabla\rho_\varepsilon)\,\nabla \rho_\varepsilon}{\rho^{p-1-\alpha}_\varepsilon},\ f_X:=(p-1-\alpha)^2\frac{F^p(\nabla \rho_\varepsilon)}{\rho^{p-\alpha}_\varepsilon}.
\]
The proof is divided into two steps.

\smallskip

\noindent \textbf{Step 1.} In this step, we show that (\ref{weightedhard}) holds if $f_X\leq \di X$ in the weak sense.

\smallskip

In fact,
if $p-1-\alpha>0$ (resp., $p-1-\alpha<0$), Theorem \ref{divlemf} (1) (resp., (2)) yields
\[
\left( \frac{|p-1-\alpha|}{p} \right)^p\int_\Omega \frac{F^p(\nabla\rho_\varepsilon)}{\rho_\varepsilon^{p-\alpha}} |{u}|^p d\mathfrak{m}\leq \int_\Omega\rho_\varepsilon^\alpha\,\max\{F^{*p}(\pm d{u})\}d\mathfrak{m}.\tag{4.2}\label{4.0}
\]

We point out that (\ref{4.0}) implies (\ref{weightedhard}).

\smallskip

\noindent \textbf{Case 1.} Suppose $\alpha\geq 0$. Since $\frac{F^p(\nabla\rho_\varepsilon)}{\rho_\varepsilon^{p-\alpha}}|{u}|^p\in L^{1}(\Omega)$ and
\[
 \rho_\varepsilon^\alpha\,\max\{F^{*p}(\pm d{u})\}\leq (\rho+1)^\alpha\,\max\{F^{*p}(\pm d{u})\}\in L^1(\Omega),
\]
(\ref{4.0}) together with Fatou's lemma   and Lebesgue's dominated convergence theorem  yields (\ref{weightedhard}). That is,
\begin{align*}
&\left( \frac{|p-1-\alpha|}{p} \right)^p\int_{\Omega} |{u}|^p  \frac{F^p(\nabla\rho)}{\rho^{p-\alpha}}d\mathfrak{m}=
\left( \frac{|p-1-\alpha|}{p} \right)^p\int_{\Omega} \underset{\varepsilon\rightarrow0^+}{\lim\inf}\left(|{u}|^p  \frac{F^p(\nabla\rho_\varepsilon)}{\rho^{p-\alpha}_\varepsilon}\right)d\mathfrak{m}\\
\leq& \underset{\varepsilon\rightarrow0^+}{\lim\inf}\left( \frac{|p-1-\alpha|}{p} \right)^p\int_{\Omega} \left(|{u}|^p  \frac{F^p(\nabla\rho_\varepsilon)}{\rho^{p-\alpha}_\varepsilon}\right)d\mathfrak{m}\leq \underset{\varepsilon\rightarrow0^+}{\lim\inf}\int_{\Omega} \rho^\alpha_\varepsilon \max\{F^*(\pm d{u}) \}d\mathfrak{m}\\
=  & \int_{\Omega} \rho^\alpha \max\{F^{*p}(\pm d{u}) \}d\mathfrak{m}.
\end{align*}

\smallskip

\noindent \textbf{Case 2.} Suppose $\alpha< 0$.
In this case, $\rho_{\varepsilon}^\alpha\,\max\{F^{*p}(\pm d{u})\}\in L^1(\Omega)$ and
\[
 \rho_{\varepsilon_2}^\alpha\,\max\{F^{*p}(\pm d{u})\}\leq \rho_{\varepsilon_1}^\alpha\,\max\{F^{*p}(\pm d{u})\},\ \text{ if } \varepsilon_1\leq \varepsilon_2.
\]
Now (\ref{4.0}) together with Fatou's lemma  and Lebesgue's monotone convergence Theorem   yields  (\ref{weightedhard}).

\smallskip

\noindent\textbf{Step 2.} From Step 1, it remains to show that  $f_X\leq \di X$ in the weak sense, that is, for any nonnegative function ${u}\in C^\infty_0(\Omega)$, one has
\[
(p-1-\alpha)^2\int_{\Omega} \frac{F^p(\nabla \rho_\varepsilon)}{\rho^{p-\alpha}_\varepsilon} {u} d\mathfrak{m}\leq (p-1-\alpha)\int_{\Omega} \frac{F^{p-2}(\nabla\rho_\varepsilon)\langle \nabla\rho_\varepsilon, d{u} \rangle}{\rho^{p-1-\alpha}_\varepsilon}d\mathfrak{m}.\tag{4.3}\label{2.18}
\]

\smallskip

In order to prove this,  choose an open set $U$ with ${\text{supp}({u})}\subset U\subset \subset \Omega$.   Let $W^{1,p}(U)$ be the completion of $C^\infty_0(U)$ with respect to the norm $\|u\|_{W^{1,p}(U)}:=\left(\int_U |u|^p d\mathfrak{m} +\int_U F^{*p}(du)d\mathfrak{m}\right)^{1/p}$. Since $\overline{U}$ is compact, $W^{1,p}(U)$ is  a  Sobolev space in the sense of Hebey \cite[Definition 2.1]{H}.

For any $k>\varepsilon$, define $\rho_{k\varepsilon}:=\inf\{\rho_\varepsilon|_{U},k\}\in L^1(U)$.
It is easy to check   $\ln\rho_{k\varepsilon}\in W^{1,p}(U)$ and hence, there is a sequence $\phi_n\in C^\infty(U)$ such that
\[
\|\phi_n-\ln\rho_{k\varepsilon}\|_{W^{1,p}(U)}\rightarrow0,\ \phi_n\rightarrow \ln \rho_{k\varepsilon} \text{ pointwise a.e.,}\ \ln\varepsilon\leq  \phi_n \leq \ln k.
\]
Let $\psi_n:=\exp(\phi_n)$. Then
$\psi_n\in C^\infty(U)$ with $\varepsilon\leq \psi_n\leq k$,
\[
\int_U |\ln \psi_n-\ln \rho_{k\varepsilon}|^p d\mathfrak{m}\rightarrow 0,\ \int_U F^{*p}\left( \frac{d \psi_n}{\psi_n}-\frac{d \rho_{k\varepsilon}}{\rho_{k\varepsilon}}  \right)d\mathfrak{m}\rightarrow 0.\tag{4.4}\label{**3.4}
\]

Now we choose $v_n:={u}/\psi_n^{p-1-\alpha}\in C^\infty_0(\Omega)$ as test functions. Since $-(p-1-\alpha)\Delta_p\rho\geq 0$ in the weak sense, one has
\[
(p-1-\alpha)^2\int_{\Omega} \frac{F^{p-2}(\nabla\rho)\,\langle \nabla\rho,d\psi_n  \rangle}{\psi^{p-\alpha}_n}{u}\, d\mathfrak{m}\leq (p-1-\alpha)\int_{\Omega}  \frac{F^{p-2}(\nabla\rho)\,\langle \nabla\rho,d{u}  \rangle}{\psi^{p-1-\alpha}_n}\, d\mathfrak{m}. \tag{4.5}\label{2.19}
\]
In the following, we derive (\ref{2.18}) from (\ref{2.19}).

\smallskip

\noindent\textbf{Case 1.} Suppose $p-1-\alpha>0$.
We study the right hand side of (\ref{2.19}) first. By
(\ref{2.2newineq}), one has
\begin{align*}
\left|\frac{F^{p-2}(\nabla\rho)}{\psi_n^{p-1-\alpha}}\langle \nabla\rho,d{u} \rangle\right|
\leq \lambda_F(U)\frac{F^{p-1}(\nabla\rho)}{\varepsilon^{p-1-\alpha}}F^*(d{u})\in L^{1}(\Omega).\tag{4.6}\label{2.7}
\end{align*}
Since $\psi_n\rightarrow \rho_{k\varepsilon}$  pointwise a.e.,  (\ref{2.7})  together with Lebesgue's dominated convergence theorem yields
\[
\lim_{n\rightarrow \infty}\int_{\Omega} \frac{F^{p-2}(\nabla\rho)}{\psi_n^{p-1-\alpha}}\langle \nabla\rho,d{u} \rangle d\mathfrak{m}=\int_{\Omega} \frac{F^{p-2}(\nabla\rho)}{\rho_{k\varepsilon}^{p-1-\alpha}}\langle \nabla\rho,d{u} \rangle d\mathfrak{m}.\tag{4.7}\label{2.9}
\]

Now we consider the left hand side of (\ref{2.19}). Firstly, an argument similar to the one above furnishes
\[
\lim_{n\rightarrow\infty}\left\|  F^{p-2}(\nabla\rho) \left\langle \nabla\rho,\frac{d\rho_{k\varepsilon}}{\rho_{k\varepsilon}}  \right\rangle  \left(\frac1{\psi_n^{p-\alpha-1}}- \frac1{\rho^{p-\alpha-1}_{k\varepsilon}} \right)                 u\right\|_{L^1(\Omega)}=0.\tag{4.8}\label{3.14new}
\]
Secondly,  (\ref{2.2newineq}) together with the H\"older inequality  and (\ref{**3.4}) implies
\begin{align*}
&\left\|\frac{F^{p-2}(\nabla\rho)}{\psi_n^{p-1-\alpha}} \left \langle \nabla\rho,\frac{d\psi_n}{\psi_n}- \frac{d\rho_{k\varepsilon}}{\rho_{k\varepsilon}}\right\rangle{u}\right\|_{L^1(\Omega)}\\
\leq &\lambda_F(U)\frac{\max_U|{u}|}{\varepsilon^{p-1-\alpha}}\|F^{p-1}(\nabla\rho)\|_{L^{p/(p-1)}(U)} \,\left\| F^*\left( \frac{d\psi_n}{\psi_n}- \frac{d\rho_{k\varepsilon}}{\rho_{k\varepsilon}} \right)\right\|_{L^p(U)}\rightarrow 0,\ n\rightarrow+\infty,
\end{align*}
which together with (\ref{3.14new}) yields
\[
\lim_{n\rightarrow \infty}\int_{\Omega} \frac{F^{p-2}(\nabla\rho)}{\psi_n^{p-\alpha}}  \langle \nabla\rho,d\psi_n\rangle   {u}          d\mathfrak{m}=\int_{\Omega} \frac{F^{p-2}(\nabla\rho)}{\rho_{k\varepsilon}^{p-\alpha}}  \langle \nabla\rho,d\rho_{k\varepsilon}\rangle {u} d\mathfrak{m}.\tag{4.9}\label{2.10}
\]

Now (\ref{2.19}) together with (\ref{2.9}) and (\ref{2.10}) yields
\[
(p-1-\alpha)^2\int_{\Omega} \frac{F^{p-2}(\nabla\rho)\,\langle \nabla\rho,d\rho_{k\varepsilon}  \rangle}{\rho_{k\varepsilon}^{p-\alpha}}{u}\, d\mathfrak{m}\leq (p-1-\alpha)\int_{\Omega}  \frac{F^{p-2}(\nabla\rho)\,\langle \nabla\rho,d{u}  \rangle}{\rho_{k\varepsilon}^{p-1-\alpha}}\, d\mathfrak{m}. \tag{4.10}\label{case1new}
\]
Due to $p-\alpha>1$, (\ref{2.2newineq}) implies
\begin{align*}
&\left|\frac{F^{p-2}(\nabla\rho)\,\langle \nabla\rho,d\rho_{k\varepsilon}  \rangle}{\rho_{k\varepsilon}^{p-\alpha}}{u}\right|\leq \frac{F^p(\nabla\rho){u}}{\varepsilon^{p-\alpha}}\in L^1(\Omega),\\
 &\left|\frac{F^{p-2}(\nabla\rho)\,\langle \nabla\rho,d{u}  \rangle}{\rho_{k\varepsilon}^{p-1-\alpha}}\right|\leq \lambda_F(U)\frac{F^{p-1}(\nabla\rho)F^*(d{u})}{\varepsilon^{p-1-\alpha}}\in L^1(\Omega),
\end{align*}
which together with (\ref{case1new}) and  Lebesgue's dominated convergence theorem yield
(\ref{2.18}).

\smallskip

\noindent \textbf{Case 2.}  Suppose $p-1-\alpha<0$. Since $1/\psi^{p-1-\alpha}\leq k^{\alpha-p+1}$, by a suitable modification to the
argument  in
Case 1, one gets  (\ref{case1new}) again. Then (\ref{2.18}) follows from a similar argument if
$p-\alpha\geq 0$. Now assume $p-\alpha<0$, in which case
we study the left hand side of (\ref{case1new}) first.
Note that
\begin{align*}
\frac{F^{p-2}(\nabla\rho)\,\langle \nabla\rho,d\rho_{k\varepsilon}  \rangle}{\rho_{k\varepsilon}^{p-\alpha}}{u}=\frac{F^{p-2}(\nabla\rho)\,\langle \nabla\rho,d\rho_{\varepsilon}  \rangle}{\rho_{k\varepsilon}^{p-\alpha}}\chi_{\{\rho_\varepsilon\leq k\}}{u}
=\frac{F^p(\nabla\rho)}{\rho_{k\varepsilon}^{p-\alpha}}\chi_{\{\rho_\varepsilon\leq k\}}{u},
\end{align*}
which implies that $ \frac{F^{p-2}(\nabla\rho)\,\langle \nabla\rho,d\rho_{k\varepsilon}  \rangle}{\rho_{k\varepsilon}^{p-\alpha}}{u} $ is an increasing sequence of nonnegative functions  converging pointwise to $\frac{F^{p-2}(\nabla\rho)\,\langle \nabla\rho,d\rho_{\varepsilon}  \rangle}{\rho_{\varepsilon}^{p-\alpha}}{u}$ as $k\rightarrow +\infty$. Hence, Lebesgue's monotone convergence theorem  furnishes
\[
\lim_{k\rightarrow+\infty}\int_{\Omega} \frac{F^{p-2}(\nabla\rho)\,\langle \nabla\rho,d\rho_{k\varepsilon}  \rangle}{\rho_{k\varepsilon}^{p-\alpha}}{u}\, d\mathfrak{m}=\int_{\Omega} \frac{F^{p-2}(\nabla\rho)\,\langle \nabla\rho,d\rho_{\varepsilon}  \rangle}{\rho_{\varepsilon}^{p-\alpha}}{u}\, d\mathfrak{m}.\tag{4.11}\label{3.******new}
\]

Now we study the right hand side of (\ref{case1new}).   The H\"older inequality together with Condition (2) (i.e., $F^p(\nabla\rho)/\rho^{p-\alpha}$, $\rho^\alpha\in L^1_{\lo}(\Omega)$) yields
\begin{align*}
\left\|\left(\frac{F^{p}(\nabla\rho)}{\rho^{{p-\alpha}}_{k\varepsilon}}\right)^{\frac1{p'}}  \left({\rho^{{\alpha}}_{k\varepsilon}}\right)^{\frac1p}\right\|_{L^1(U)}
\leq \left\|\frac{F^{p}(\nabla\rho)}{\rho^{{p-\alpha}}_{k\varepsilon}} \right\|^{\frac1{p'}}_{L^{1}(U)}\left\|  {\rho^{{\alpha}}_{k\varepsilon}} \right\|^{\frac1p}_{L^1(U)}
<+\infty,
\end{align*}
where $p'=p/(p-1)$.
Therefore, we have
\begin{align*}
\left|  \frac{F^{p-2}(\nabla\rho)\,\langle \nabla\rho,d{u}  \rangle}{\rho_{k\varepsilon}^{p-1-\alpha}}\right|&\leq \lambda_F(U)\, F^*(d{u})\,\frac{F^{p-1}(\nabla\rho)}{\rho^{\frac{p-\alpha}{p'}}_{k\varepsilon}}\frac{1}{\rho^{-\frac{\alpha}{p}}_{k\varepsilon}}\leq C\left(\frac{F^{p}(\nabla\rho)}{\rho^{{p-\alpha}}_{k\varepsilon}}\right)^{\frac1{p'}}  \left({\rho^{{\alpha}}_{k\varepsilon}}\right)^{\frac1p}\in L^1(U),
\end{align*}
which together with Lebesgue's dominated convergence theorem   furnishes
\[
\lim_{k\rightarrow+\infty}\int_{\Omega}  \frac{F^{p-2}(\nabla\rho)\,\langle \nabla\rho,d{u}  \rangle}{\rho_{k\varepsilon}^{p-1-\alpha}}\, d\mathfrak{m}=\int_{\Omega}  \frac{F^{p-2}(\nabla\rho)\,\langle \nabla\rho,d{u}  \rangle}{\rho_{\varepsilon}^{p-1-\alpha}}\, d\mathfrak{m}.\tag{4.12}\label{2.24}
\]
Now (\ref{2.18}) follows from  (\ref{case1new}),  (\ref{3.******new}) and (\ref{2.24}).
\end{proof}

\begin{remark}\label{th1rem}
For $p>\alpha$, the above theorem implies $\frac{F^p(\nabla \rho)}{\rho^{p-\alpha}}\in L^1_{\loc}(\Omega)$ naturally.
\end{remark}

\subsection{Best constant and Brezis-V\'azquez improvement}
Suppose the assumption of Theorem \ref{wegHardin} holds and additionally assume that $p-1-\alpha\neq0$, $\rho>0$ a.e.  and $d\rho\neq0$ a.e..
Then one can define a norm on $C^\infty_0(\Omega)$ by
\[
|u|_D:=\left( \int_\Omega \rho^\alpha\max\{F^{*p}(\pm d{u})\}d \mathfrak{m}\right)^{\frac1p}.
\]
Denote by $D^{1,p}(\Omega,\rho^\alpha)$  the closure of  $C^\infty_0(\Omega)$ with respect to the norm $|\cdot|_D$.
\begin{remark}
In fact, $D^{1,p}(\Omega,\rho^\alpha)$ is  the completion of  $C^\infty_0(\Omega)$ with respect to the norm
\[
\|u\|_D:=\left(\int_{\Omega} \frac{F^p(\nabla\rho)}{\rho^{p-\alpha}}|u|^pd\mathfrak{m}+ \int_\Omega \rho^\alpha\max\{F^{*p}(\pm d{u})\}d \mathfrak{m}\right)^{\frac1p}.
\]
In view of Theorem \ref{wegHardin}, $|\cdot|_D$ is equivalent to $\|\cdot\|_D$.
\end{remark}

\begin{proposition}\label{truebestconstant}Let $(M,F,d\mathfrak{m})$ be a forward complete FMMM and let $\Omega\subset M$ be a natural domain.
Given $p>1$ with $p-1-\alpha\neq0$, suppose $\rho\in W^{1,p}_{\lo}(\Omega)$ satisfies the following properties:

\smallskip

(i) $-(p-1-\alpha)\Delta_p\rho\geq0$ in the weak sense;

(ii) $\frac{F^p(\nabla \rho)}{\rho^{p-\alpha}}$, $\rho^\alpha\in L^1_{\loc}(\Omega)$ if $\alpha>p$.

(iii) $\rho>0$ a.e. and $d\rho\neq0$ a.e. with $F^*(d\rho)\geq F^*(-d\rho)$;

\smallskip

\noindent Set
\[
\mathfrak{C}_{p,\alpha}(\Omega):=\inf_{u\in D^{1,p}(\Omega,\rho^\alpha)\backslash\{0\}}\frac{\int_{\Omega} \rho^\alpha\max\{F^{*p}(\pm du)\}d\mathfrak{m}}{\int_{\Omega} \frac{F^p(\nabla\rho)}{\rho^{p-\alpha}}|u|^pd\mathfrak{m}}.
\]
Thus,  if $\rho^\frac{p-1-\alpha}p\in D^{1,p}(\Omega,\rho^\alpha)$, then
\[
\mathfrak{C}_{p,\alpha}(\Omega)=\left(\frac{|p-1-\alpha|}{p}\right)^p
\] and $C\rho^\frac{p-1-\alpha}p$ is an extremal, where $C\in \mathbb{R}\backslash\{0\}$.
\end{proposition}

\begin{proof} Theorem \ref{wegHardin} implies $\mathfrak{C}_{p,\alpha}(\Omega)\geq \left(\frac{|p-1-\alpha|}{p}\right)^p$.
Now set $\varphi=C\rho^{\frac{p-\alpha-1}p}$. A direct calculation together with the assumption yields
\[
\rho^\alpha\max\{F^{*p}(\pm d\varphi)\}=|C|^p\left(\frac{|p-\alpha-1|}{p}\right)^p\frac{F^{*p}(d\rho)}{\rho}\in L^1(\Omega),
\]
which furnishes
\begin{align*}
\int_{\Omega}\rho^\alpha\max\{F^{*p}(\pm d\varphi)\}d\mathfrak{m}=\left(\frac{|p-\alpha-1|}{p}\right)^p\int_{\Omega} \frac{F^p(\nabla\rho)}{\rho^{p-\alpha}}|\varphi|^p d\mathfrak{m}.
\end{align*}
Hence, it follows that $\mathfrak{C}_{p,\alpha}(\Omega)\leq \left(\frac{|p-1-\alpha|}{p}\right)^p$.
\end{proof}

We also have the following Brezis-V\'azquez improvement for $p=2$ and $\alpha=0$.
\begin{proposition}Let $(M,F,d\mathfrak{m})$ be a forward complete FMMM with finite uniformity constant $\Lambda_F$ and let $\Omega\subset M$ be a natural domain. Suppose that $\rho\in W^{1,2}_{\lo}(\Omega)$ satisfies
$\rho>0$ a.e.  and $-\Delta \rho\geq 0$ in the weak sense.
Set   $\Theta(\Omega)$ as
\[
\Theta(\Omega):=\inf_{u\in C^1_0(\Omega)\backslash\{0\}}\frac{\int_{\Omega} \rho  \min\{F^{*2}(\pm du)\}d\mathfrak{m}}{\int_{\Omega}\rho\, u^2 d\mathfrak{m}}.
\]
Then we have
\begin{align*}
\int_{\Omega}\max\{F^{*2}(\pm du)\} d\mathfrak{m}\geq&  \frac{1}4  \int_{\Omega}  \frac{u^2}{\rho^2}F^{2}(\nabla\rho) d\mathfrak{m}+\frac{\Theta(\Omega)}{\Lambda_F}\int_{\Omega} {u^2}d\mathfrak{m}, \ \forall\,u\in C^\infty_0(\Omega).\tag{4.13}\label{new3.16p2}
\end{align*}
In particular, if $\rho^{1/2}\notin D^{1,2}(\Omega)$ but $\mathfrak{C}_{2,0}(\Omega)=1/4$, then the best constant $\mathfrak{C}_{2,0}(\Omega)$ is not achieved.
\end{proposition}
\begin{proof}
Given $u\in C^1_0(\Omega)$ and $\varepsilon\in (0,1)$, set $\rho_\varepsilon:=\rho+\varepsilon$ and $v:=u/\rho^{\frac12}_\varepsilon$. Let $\Omega_0:=\{x\in \Omega:\, v(x)=0\}$, $\Omega_+:=\{x\in \Omega:\, v(x)>0\}$, $\Omega_-:=\{x\in \Omega:\, v(x)<0\}$.
Set $\xi:=\frac12 v \rho_\varepsilon^{-\frac12}d\rho_\varepsilon$ and $\eta:=\rho_\varepsilon^{\frac12} dv$. Thus, $du=\xi+\eta$.
Using (\ref{ineq}), on ${\Omega_+}\cup {\Omega_0}$ we have
\begin{align*}
\max\{F^{*2}(\pm du)\}\geq F^{*2}(du)\geq \frac14 v^2 {\rho_\varepsilon}^{-1}F^{*2}(d{\rho_\varepsilon})+  vg^*_{d{\rho_\varepsilon}}(d{\rho_\varepsilon},dv)+\frac{{\rho_\varepsilon}}{\Lambda_F}F^{*2}(dv),
\end{align*}
which yields
\begin{align*}
&\int_{{\Omega_+}\cup {\Omega_0}}\max\{F^{*2}(\pm du)\} d\mathfrak{m}-\frac14\int_{{\Omega_+}\cup {\Omega_0}}  \frac{u^2}{{\rho_\varepsilon}^2}F^{*2}(d{\rho_\varepsilon}) d\mathfrak{m}\\
\geq &\frac12\int_{{\Omega_+}\cup {\Omega_0}}\langle\nabla{\rho}, dv^2\rangle d\mathfrak{m}+\frac{1}{\Lambda_F}\int_{{\Omega_+}\cup {\Omega_0}}{\rho_\varepsilon}F^{*2}\left( dv  \right)d\mathfrak{m}.\tag{4.14}\label{2.15}
\end{align*}
A similar argument on ${\Omega_-}$ furnishes
\begin{align*}
&\int_{{\Omega_-}}\max\{F^{*2}(\pm du)\} d\mathfrak{m}-\frac14\int_{{\Omega_-}}  \frac{u^2}{{\rho_\varepsilon}^2}F^{*2}(d{\rho_\varepsilon}) d\mathfrak{m}\\
\geq&\frac12\int_{{\Omega_-}}\langle\nabla{\rho}, dv^2\rangle d\mathfrak{m}+\frac{1}{\Lambda_F}\int_{{\Omega_-}}{\rho_\varepsilon} F^{*2}\left( -dv  \right)d\mathfrak{m},
\end{align*}
which together with
 (\ref{2.15}) and $-\Delta\rho\geq 0$ (in the weak sense) yields
\begin{align*}
&\int_{\Omega}\max\{F^{*2}(\pm du)\} d\mathfrak{m}-\frac14\int_{\Omega}  \frac{u^2}{{\rho_\varepsilon}^2}F^{*2}(d{\rho}) d\mathfrak{m}\\
\geq&\frac{1}{\Lambda_F}\int_{\Omega}{\rho_\varepsilon} F^{*2}\left( d|v|  \right)d\mathfrak{m}\geq
\frac{1}{\Lambda_F}\int_{\Omega}{\rho} \min\{F^{*2}\left(\pm dv  \right)\}d\mathfrak{m}\geq \frac{\Theta(\Omega)}{\Lambda_F}\int_{\Omega} \frac{\rho}{\rho_\varepsilon}u^2d\mathfrak{m} .\tag{4.15}\label{3.22new}
\end{align*}
Note that Remark \ref{th1rem} implies $F^2(\nabla\rho)/\rho^2\in L^1_{\lo}(\Omega)$ and hence,
\begin{align*}
\left|  \frac{u^2}{{\rho_\varepsilon}^2}F^{*2}(d{\rho})  \right|&\leq \frac{u^2}{{\rho}^2}F^{*2}(d{\rho})\in L^1(\Omega),
\end{align*}
which together with Lebesgue's dominated convergence theorem  and (\ref{3.22new}) yields (\ref{new3.16p2}).

Now suppose $\rho^{1/2}\notin D^{1,2}(\Omega)$ but $\mathfrak{C}_{2,0}(\Omega)=1/4$. Thus, from (\ref{3.22new}) and Lebesgue's dominated convergence theorem, we have
\begin{align*}
\int_{\Omega}\max\{F^{*2}(\pm du)\} d\mathfrak{m}-\frac14\int_{\Omega}  \frac{u^2}{{\rho}^2}F^{*2}(d{\rho}) d\mathfrak{m}\geq \frac{1}{\Lambda_F}\int_{\Omega}{\rho} F^{*2}(d|v|)d\mathfrak{m}>0,
\end{align*}
which implies the nonexistence of minimizers in $D^{1,2}(\Omega)$.
\end{proof}

%\begin{remark}
%Although the results in this paper are mainly concerned with forward complete Finsler metric
%measure manifolds, all of them can be generalized to the backward complete case by an easy argument
%based on a reverse Finsler manifold. We leave the formulation of such statements to the interested reader.
%\end{remark}

\appendix

\section{Two lemmas}\label{Aapp}

\begin{lemma}\label{lpschcom}
Let  $(M,o,F,d\mathfrak{m})$, $\Omega$, $p,\beta$   be as in Definition \ref{DefDS}. If $u$ is a globally Lipschitz function on $M$  with compact support in $\Omega$, then $u\in D^{1,p}(\Omega,r^{p+\beta})$.
\end{lemma}
\begin{proof}
Since $\text{supp}(u)$ is compact, there exist  a coordinate covering   $\{(U_k,\phi_k)\}_{k=1}^{N<\infty}$ of $\text{supp}(u)$ and a constant $C\geq 1$ such that for each $k$,  $U_k\subset\subset \Omega$, $\phi_k(U_k)=\mathbb{B}_{\mathbf{0}}(1)$ and
\begin{align*}
C^{-1} d\vol\leq d\mathfrak{m}|_{U_k}\leq C d\vol,\ \ \  C^{-1}\leq \frac{F^*(\omega)}{\|\phi^{-1}_{k*}\omega\|}\leq C, \ \text{ for any }\omega\in T^*{U_k}\backslash\{0\},\tag{A.1}\label{newA.1constant}
\end{align*}
where $d\vol$ and $\|\cdot\|$ are the Lebesgue measure and the Euclidean norm on the unit  ball $\mathbb{B}_{\mathbf{0}}(1)$, respectively.

Choose a number $q>1$ such that $\beta q/({q-1})>-n$ if  $\beta>-n$.  By Lemma \ref{centerinteg} and the construction above, one can easily verify $\int_{U_k}r^{{\beta q}/({q-1})} d \mathfrak{m}<\infty$ for each $k$.

On the other hand,
let $\{\eta_k\}$ be a smooth partition of unity subordinate to $\{U_k\}$. Thus,  $(\eta_ku)\circ\phi_k^{-1}$ is a globally Lipschitz function on $\mathbb{B}_{\mathbf{0}}(1)$ with respect to the Euclidean distance and hence, $(\eta_ku)\circ\phi_k^{-1}$ belongs to the  Sobolev space $ W^{1,pq}(\mathbb{B}_{\mathbf{0}}(1))$.
Meyers-Serrin's theorem then yields a sequence $v_{k_j}\in C_0^\infty(\mathbb{B}_{\mathbf{0}}(1))$ with $\lim_{j\rightarrow+\infty}\|v_{k_j}-(\eta_ku)\circ\phi_k^{-1}\|_{W^{1,pq}(\mathbb{B}_{\mathbf{0}}(1))}=0$. Therefore, we have $v_{k_j}\circ\phi_k\in C^\infty_0(\Omega)$ with $\text{supp}(v_{k_j}\circ\phi_k)\subset U_k$, which together with the H\"older inequality and (\ref{newA.1constant}) implies
\begin{align*}
\int_\Omega|v_{k_j}\circ\phi_k-(\eta_ku)|^p r^{\beta} d\mathfrak{m}&\leq \left(C\int_{\mathbb{B}_{\mathbf{0}}(1)} |v_{k_j}-(\eta_ku)\circ \phi^{-1}_k |^{pq} d \vol\right)^{\frac1q} \left(  \int_{U_k}r^{\frac{\beta q}{q-1}} d \mathfrak{m}\right)^{\frac{q-1}q}\\
&\leq C^{\frac1q}\|v_{k_j}-(\eta_ku)\circ\phi_k^{-1}\|^p_{W^{1,pq}(\mathbb{B}_{\mathbf{0}}(1))}\left(  \int_{U_k}r^{\frac{\beta q}{q-1}} d \mathfrak{m}\right)^{\frac{q-1}q}\rightarrow 0, \text{ as }j\rightarrow+\infty.
\end{align*}
Similarly, one can prove $\lim_{j\rightarrow+\infty}\int_\Omega F^{*p}(d(v_{k_j}\circ\phi_k)-d(\eta_ku)) r^{\beta+p} d\mathfrak{m}=0$. Therefore, $\|v_{k_j}\circ\phi_k-(\eta_ku)\|_D\rightarrow 0$ and $(\eta_ku)\in D^{1,p}(\Omega,r^{p+\beta})$. We conclude the proof by $u=\sum_{k=1}^N(\eta_ku)$.
\end{proof}

\begin{lemma}\label{compinft}

Let $(M,o,F,d\mathfrak{m})$ be an $n$-dimensional forward complete PFMMM with
 $\mathbf{Ric}_\infty\geq (n-1)K$ and $\mathbf{S}_o^+\geq -a$, where $a\geq 0$.  Set $\frac{\pi}{\sqrt{K}}:=+\infty$ if $K\leq 0$. Let $(t,y)$ denote the polar coordinate system around $o$.
Then we have
\[
\Delta t\leq (n-1)\frac{\mathfrak{s}'_K(t)}{\mathfrak{s}_K(t)}+a,\ \text{ for any }y\in S_oM,\  0<t<\min\left\{i_y, \frac{\pi}{2\sqrt{K}} \right\},
\]
which implies
\[
\hat{\sigma}_o(t,y)\leq e^{-\tau(y)+at}\mathfrak{s}^{n-1}_{K}(t),\ \text{ for any }y\in S_oM,\  0<t<\min\left\{i_y, \frac{\pi}{2\sqrt{K}} \right\}.%\tag{A.1}\label{volumA.1}
\]
where $\hat{\sigma}_o(t,y)$ is defined as in (\ref{formulavolume}) and $\mathfrak{s}_K(t)$ is the solution of $f''+K f=0$ with $f(0)=0$ and $f'(0)=1$. Hence,
\[
\mathfrak{m}(B^+_o(r))\leq \int_{S_oM}e^{-\tau(y)}d\nu_o(y)\int_0^{r} e^{at}\mathfrak{s}^{n-1}_{K}(t)dt, \text{ for } 0<r<\min\left\{i_y, \frac{\pi}{2\sqrt{K}} \right\}.
\]
\end{lemma}
\begin{proof}
The proof is similar to that of Wei and Wylie \cite[Theorem 1.1]{WW22}. First, fix $y\in S_oM$ and set
\[
H(t):=\frac{\partial}{\partial t}\log\sqrt{\det g_{\nabla t}}, \ \tau(t):=\tau(\nabla t),\  \mathbf{S}(t):=\mathbf{S}(\nabla t)=\frac{d}{dt}\tau(t).
\]
A standard argument (cf. Wu \cite[(4.5)]{Wu2}) yields $\Delta t=H(t)-\mathbf{S}(t)$ and
\[
\frac{d}{d t}H\leq -\mathbf{Ric}(\nabla t)-\frac{H^2}{n-1}.\tag{A.2}\label{A.1}
\]
Also set $H_K(t):=(n-1)\frac{\mathfrak{s}'_K(t)}{\mathfrak{s}_K(t)}$.
By (\ref{A.1}), one has
\begin{align*}
&\frac{d}{d t}\left(H(t)\mathfrak{s}^2_K(t)\right)\leq 2 \mathfrak{s}'_K(t)\mathfrak{s}_K(t)H(t)-\mathfrak{s}^2_K(t)\left( \mathbf{Ric}(\nabla t)+\frac{H^2}{n-1} \right)\\
=&-\left( \frac{\mathfrak{s}_K(t) H(t)}{\sqrt{n-1}}-\sqrt{n-1}\,\mathfrak{s}'_K(t) \right)^2+(n-1)(\mathfrak{s}'_K(t))^2-\mathfrak{s}^2_K(t)\mathbf{Ric}(\nabla t)\\
\leq& (n-1)(\mathfrak{s}'_K(t))^2-\mathfrak{s}^2_K(t)\left( \mathbf{Ric}_\infty(\nabla t)-\frac{d}{dt}\mathbf{S}(t) \right)\\
\leq&(n-1)(\mathfrak{s}'_K(t))^2-(n-1)K\mathfrak{s}^2_K(t)+\mathfrak{s}^2_K(t)\frac{d}{dt}\mathbf{S}(t)\\
=&\frac{d}{dt}\left(H_K(t)\mathfrak{s}^2_K(t)\right)+\mathfrak{s}^2_K(t)\frac{d}{dt}\mathbf{S}(t),
\end{align*}
Since $\mathbf{S}_O\geq -a$,   integrating by parts on the above inequality, we get
\begin{align*}
&\mathfrak{s}^2_K(t)\Delta t=\mathfrak{s}^2_K(t)(H(t)-\mathbf{S}(t))\leq\mathfrak{s}^2_K(t)H_K(t)-\int^t_0\frac{d}{ds}(\mathfrak{s}^2_K(s)) \,\mathbf{S}(s) ds\tag{A.3}\label{A.2}\\
\leq & \mathfrak{s}^2_K(t)H_K(t)+a\int^t_0\frac{d}{ds}(\mathfrak{s}^2_K(s))ds=\mathfrak{s}^2_K(t)H_K(t)+a\mathfrak{s}^2_K(t).
\end{align*}
Hence, $\Delta t\leq H_K(t)+a$, which implies
\begin{align*}
\frac{\partial}{\partial t}\log \hat{\sigma}_o(t,y)=\Delta t\leq (n-1)\frac{\mathfrak{s}'_K(t)}{\mathfrak{s}_K(t)}+a=\frac{\partial}{\partial t}\log \left[  e^{at}\mathfrak{s}^{n-1}_K(t) \right].
\end{align*}
Then the estimates of $\hat{\sigma}_o(t,y)$ and $\mathfrak{m}(B^+_o(r))$  follow from a standard argument (cf. Zhao et al. \cite{ZS}).
\end{proof}

\begin{remark}By a different method, Yin \cite{Y} obtained the theorem above   in the case when the PFMMM is equipped by the Busemann-Hausdroff measure and satisfies $\mathbf{Ric}_\infty\geq (n-1)K$ and  $\mathbf{S}_o^+\geq -a$ ($a>0$).
\end{remark}

\section{Weighted Sobolev space}\label{SObespa}
Let $(M,o,F,d\mathfrak{m})$ be an $n$-dimensional {\it closed} reversible RFMMM and set $r(x):=d_F(o,x)$.
Given $p\in (1,+\infty)$ and $\beta<-n$ with $p+\beta>-n$, by Lemma \ref{centerinteg},  we define a norm on $C^\infty_0(M)=C^\infty(M)$ as
\[
\|u\|_{p,\beta}:=\left( \int_{{M}}|u|^p{r}^{p+\beta} d\mathfrak{m}+\int_{{M}} F^{*p}(d u) {r}^{p+\beta} d\mathfrak{m}\right)^{\frac1p}.
\]
The {\it weighted Sobolev space} $W^{1,p}({M}, {r}^{p+\beta})$  is defined as
\[
W^{1,p}({M}, {r}^{p+\beta}):=\overline{{C}^\infty_0({M})}^{\|\cdot\|_{p,\beta}}.
\]
In particular, $W^{1,p}(M,r^0)=:W^{1,p}(M)$, i.e.,  the standard Sobolev space
in the sense of Hebey \cite[Definition 2.1]{H}.

We also define the weighted $L^p$-space $L^p(M,r^{p+\beta})$ (resp., $L^p(TM,r^{p+\beta})$) as the completion of $C^\infty(M)$ (resp., $\Gamma^\infty(T^*M)$, i.e., the space of the smooth sections of the cotangent bundle) under the norm
\[
[u]_{p,\beta}:=\left( \int_{{M}}|u|^p{r}^{p+\beta} d\mathfrak{m}\right)^{\frac1p} \ \ \left(\text{resp., }  [\omega]_{p,\beta}:=\left( \int_{{M}}F^{*p}(\omega){r}^{p+\beta} d\mathfrak{m}\right)^{\frac1p} \right).
\]
And set $L^p(M):=L^p(M,r^0)$ and $L^p(TM):=L^p(TM,r^0)$.
%\begin{remark}
%Although we begin with closed manifold, all the arguments in this subsection remain valid in the non-compact  complete case. For example, let $(M,F,d\mathfrak{m})$ be the standard Eucildean space $(\mathbb{R}^n,|\cdot|,dx)$. Then the results here are the standard theory  (cf. \cite{H,HKM,Kt}).
%\end{remark}

\begin{lemma}\label{Soleweakder}
If $u\in W^{1,p}({M},{r}^{p+\beta})$, then $u\in W^{1,1}(M)$. Moreover, the differential $\varpi$ of $u$ in $W^{1,p}({M},{r}^{p+\beta})$ is the distributional derivative of $u$, i.e.,  $\varpi\in L^1(TM)$ and
\[
\int_{{M}} \langle X,\varpi\rangle d\mathfrak{m}=-\int_{{M}}u \di X d\mathfrak{m},\ \text{ for any smooth vector field }X.
\]
\end{lemma}
\begin{proof}Since $({\beta+p})/({1-p})>-n$, Lemma \ref{centerinteg} implies that  ${r}^{\frac{\beta+p}{1-p}}$ is integrable.
Given $f\in L^p({M},{r}^{\beta+p})$, the H\"older inequality yields
\begin{align*}
\int_M|f| d\mathfrak{m}&=\int_M |f|{r}^{\frac{\beta+p}{p}}{r}^{-\frac{\beta+p}{p}}d\mathfrak{m}\leq \left( \int_M |f|^p {r}^{p+\beta}d\mathfrak{m} \right)^{\frac1p}\left( \int_M {r}^{\frac{p+\beta}{1-p}} d\mathfrak{m} \right)^{\frac{p-1}p}. \tag{B.1}\label{newappedix}
\end{align*}
Consequently, if $u\in W^{1,p}({M},{r}^{p+\beta})$,   (\ref{newappedix}) implies $u\in L^{1}({M})$ and its differential $\varpi\in L^{1}(T{M})$.
On the other hand,
there exist a sequence $u_j\in C^\infty_0({M})$ such that $[u_j-u]_{p,\beta}\rightarrow0$ and $[du_j-\varpi]_{p,\beta}\rightarrow0$. Thus, for any smooth vector field $X$,  (\ref{newappedix}) together with the compactness of $M$ yields
\begin{align*}
&\left| \int_{{M}} \langle  X,\varpi\rangle-(-u\di X) d\mathfrak{m} \right|=\left| \int_{{M}}  \langle  X,\varpi\rangle- \langle X,du_j\rangle+ \langle X,du_j\rangle -(-u\di X) d\mathfrak{m}  \right|\\
\leq &\int_M \left|\langle  X, \varpi-du_j\rangle \right|d\mathfrak{m}+ \int_M \left|(u_j-u)\di X \right|d\mathfrak{m}\\
\leq &\max_M F(X) \int_{M}F^*(\varpi-du_j)d\mathfrak{m} +\max_M|\di X| \int_{M}|u_j- u|d\mathfrak{m} \\
\leq &\left(\max_M F(X)+\max_M|\di X|\right)\left( \int_M {r}^{\frac{p+\beta}{1-p}} d\mathfrak{m} \right)^{\frac{p-1}p}\left([\varpi -d u_j]_{p,\beta}^p+[u_j -u]_{p,\beta}^p \right)\rightarrow 0.
\end{align*}
Furthermore, (\ref{newappedix}) also implies that $u_j\rightarrow u$ in $W^{1,1}({M})$ and hence, the lemma follows.
\end{proof}

%\begin{remark}
%For the complete Riemannian case, if $u\in C^\infty_0(M)$, then $\max\{u,0\}$ is a Lipsichitz function and hence, $u_+,u_-,|u|\in W^{1,p}(M)$. For any $u\in W^{1,p}(M)$, use the triangle inequality as above, we see  $u_+,u_-,|u|\in W^{1,p}(M)$.
%\end{remark}

\begin{lemma}\label{maxminle}
If $u\in W^{1,p}({M},{r}^{\beta+p})$, then $u_+:=\max\{u,0\}$, $u_-:=-\min\{u,0\}$ and $|u|=u_+-u_-$ are all in $W^{1,p}({M},{r}^{\beta+p})$.
\end{lemma}
\begin{proof}Since $u_-=(-u)_+$,
it suffices to prove $u_+\in W^{1,p}({M},{r}^{\beta+p})$.
Choose a sufficiently large constant $q>1$ such that $\frac{q(\beta+p)}{q-1}>-n$. For any   $f\in L^{pq}(M)$, the H\"older inequality together with Lemma \ref{centerinteg} yields
\begin{align*}
\int_M |f|^p{r}^{\beta+p}d\mathfrak{m}\leq \left( \int_M  |f|^{pq} d\mathfrak{m}\right)^{\frac1q} \left(  \int_M {r}^{\frac{q(\beta+p)}{q-1}}d\mathfrak{m}\right)^{\frac{q-1}{q}}.\tag{B.2}\label{newholderB2}
\end{align*}

First we consider the case when $u\in C^\infty_0({M})$.
The standard theory yields a subsequence $u_j\in C^\infty_0({M})$ such that
$u_j\rightarrow u_+$ in $W^{1,pq}({M})$ (cf. Hebey \cite[Lemma 2.5]{H}),
which together with (\ref{newholderB2}) implies  $u_j\rightarrow u_+$ in $W^{1,p}({M}, {r}^{p+\beta})$. Hence, $u_+\in W^{1,p}({M}, {r}^{p+\beta})$.

For the general case (i.e., $u\in W^{1,p}({M},{r}^{\beta+p})$), choose a sequence $u_j\in C^\infty_0({M})$ such that $\|u_j- u\|_{p,\beta}\rightarrow 0$. From above, we have $u_{j+}=\max\{u_j,0\}\in W^{1,p}({M}, {r}^{p+\beta})$. Since $\max\{s,t\}=\frac12(s+t-|s-t|)$,  the triangle inequality yields $\|u_{j+}-u_+\|_{p,\beta}\leq \|u_j-u\|_{p,\beta}\rightarrow 0$. Hence, $u_+\in W^{1,p}({M},{r}^{\beta+p})$.
\end{proof}

Since $M$ is closed, the following result follows from Lemma \ref{maxminle} directly.
\begin{corollary}\label{corominmax}
Given $u\in W^{1,p}({M},{r}^{\beta+p})$, then $u_\epsilon:=\max\{u-\epsilon,0\}\in W^{1,p}({M},{r}^{\beta+p})$, for any $\epsilon>0$.
\end{corollary}

Now set ${M_o}:=M\backslash\{o\}$.
Define the {\it weighted Sobolev space} $W^{1,p}({M_o},r^{\beta+p})$ as the completion of $C^\infty_0({M_o})$ with respect to the norm
\[
\|u\|_{{M_o},p,\beta}:=\left( \int_{{M_o}}|u|^p{r}^{p+\beta} d\mathfrak{m}+\int_{{M_o}} F^{*p}(d u)^p {r}^{p+\beta} d\mathfrak{m}\right)^{\frac1p}.
\]

\begin{lemma}\label{easylemma1}
If $u\in W^{1,p}({M},{r}^{\beta+p})$ with compact support in ${M_o}$, then $u|_{{M_o}}\in W^{1,p}({M_o},{r}^{\beta+p})$.
\end{lemma}
\begin{proof}
Since $\text{supp}(u)\subset {M_o}$ is compact, one can choose a cut-off function $\eta\in C^\infty_0({M})$ such that $\text{supp}(u)\subsetneqq \text{supp}(\eta)\subset {M_o}$ and $\eta|_{\text{supp}(u)}=1$.

On the other hand, since $u\in W^{1,p}({M},{r}^{\beta+p})$, there exist a sequence $u_i\in C^\infty_0({M})$ with $\|u_i- u\|_{p,\beta}\rightarrow 0$. Note that if $\|\eta u_i- \eta u\|_{{M_o},p,\beta}\rightarrow0$, then
$u|_{{M_o}}=\eta u\in W^{1,p}({M_o},{r}^{\beta+p})$ and the lemma follows. Hence, it suffices to show $\|\eta u_i- \eta u\|_{{M_o},p,\beta}\rightarrow0$.

A direct calculation together with the triangle inequality (i.e., $F^*(\omega_1+\omega_2)\leq F^*(\omega_1)+F^*(\omega_2)$) furnishes
\begin{align*}
&\|\eta u_i-\eta u\|_{{M_o},p,\beta}^p
= \int_{\text{supp}\eta}|\eta u_i-\eta u|^p {r}^{\beta+p} d\mathfrak{m}+\int_{\text{supp}\eta}F^{*p}\left(d(\eta u_i-\eta u)\right) {r}^{\beta+p}d\mathfrak{m}\\
\leq &\int_{{M}}| u_i- u|^p {r}^{\beta+p} d\mathfrak{m}+2^p\left[ \int_{\text{supp}\eta} F^{*p}\left((u_i- u)d\eta\right) {r}^{\beta+p}d\mathfrak{m}+\int_{\text{supp}\eta}F^{*p}\left(\eta d(u_i-u)\right){r}^{\beta+p}d\mathfrak{m}  \right]\\
\leq &\int_{{M}}| u_i- u|^p {r}^{\beta+p} d\mathfrak{m}+2^p\left[ \|F^{*p}(d\eta)\|_{\infty}\int_{M} |u_i- u|^p {r}^{\beta+p}d\mathfrak{m}+\int_{M}F^{*p}(du_i-du){r}^{\beta+p}d\mathfrak{m}  \right]\\
\leq&\left[2^p (\|F^{*p}(d\eta)\|_{\infty}+1)+1  \right]\|u_i-u\|_{p,\beta}^p\rightarrow 0.
\end{align*}
\end{proof}

\begin{proof}[Proof of Proposition \ref{thsob}]Without loss of generality, we may prove the proposition in the case when $u\geq 0$.
Thus, Lemma \ref{maxminle} implies $u=u_+\in  W^{1,p}({M}, {r}^{\beta+p})\cap C({M})$.

For each $\epsilon\in (0,1)$, set $u_\epsilon(x):=\max\{u-\epsilon,0\}$. Since  $u$ is continuous with $u(o)=0$, there exists a small $\delta>0$ such that $u_\epsilon=0$ in $B_o(\delta)$, which implies that $\text{supp}(u_\epsilon)$ is a compact subset of ${M_o}$. Corollary \ref{corominmax} then yields
 $u_\epsilon|_{{M_o}} \in W^{1,p}({M_o},{r}^{\beta+p})$.
By a direct calculation, we have
\begin{align*}
&\left\|u|_{{M_o}}-u_\epsilon|_{{M_o}}\right\|_{{M_o},p,\beta}^p=\left\|  u-  u_\epsilon   \right\|^p_{p,\beta}=\int_{{M}}|u-u_\epsilon|^p{r}^{p+\beta} d\mathfrak{m} + \int_{{M}} F^{*p}(d (u-u_\epsilon)) {r}^{p+\beta} d\mathfrak{m}\\
\leq & \epsilon^p \int_M {r}^{\beta+p} d\mathfrak{m}+\int_{{M}}\chi_{\{0\leq u\leq \epsilon\}}|u|^p{r}^{\beta+p} d\mathfrak{m} + \int_{{M}} \chi_{\{0\leq u\leq \epsilon\}} F^{*p}(du) {r}^{p+\beta} d\mathfrak{m}.%\tag{B.3}\label{B.3new}
\end{align*}
Now the assumption together with
the dominated convergence theorem yields
\begin{align*}
&\lim_{\epsilon\rightarrow 0^+}\int_{{M}}\chi_{\{0\leq u\leq \epsilon\}}|u|^p{r}^{\beta+p} d\mathfrak{m}=\int_{{M}}\lim_{\epsilon\rightarrow 0^+}\chi_{\{0\leq u\leq \epsilon\}}|u|^p{r}^{\beta+p} d\mathfrak{m}=0,\\
&\lim_{\epsilon\rightarrow 0^+}\int_{{M}} \chi_{\{0\leq u\leq \epsilon\}} F^{*p}(du) {r}^{p+\beta} d\mathfrak{m}=\int_{{M}} \lim_{\epsilon\rightarrow 0^+}\chi_{\{0\leq u\leq \epsilon\}} F^{*p}(du) {r}^{p+\beta} d\mathfrak{m}=0,
\end{align*}
which imply $\left\|u|_{{M_o}}-u_\epsilon|_{{M_o}}\right\|_{{M_o},p,\beta}^p\rightarrow 0$ as $\epsilon\rightarrow 0^+$ and hence, $u|_{{M_o}}\in  W^{1,p}({M_o},{r}^{\beta+p})$.
\end{proof}

\noindent{\textbf{Acknowledgements}}
 This work was supported by NNSFC (No. 11761058) and NSFS (No. 19ZR1411700).  The author is greatly indebted to Pro. A. Krist\'aly for many
useful discussions and  helpful comments.


\begin{thebibliography}{10}

\bibitem{ACR} Adimurthi, N. Chaudhuri, N. Ramaswamy, \textsl{An improved Hardy Sobolev inequality and its applications}.
Proc. Amer. Math. Soc. \textbf{130} (2002), 489--505.


\bibitem{AlB} J. Alvarez-Paiva, G. Berck, \textsl{What is wrong with the
	Hausdorff measure in Finsler spaces}. Adv.  Math.
{\textbf{204}} (2006), 647--663.

\bibitem{AlT}J. Alvarez-Paiva, A. C. Thompson, \textsl{Volumes in normed and Finsler
	spaces}, A Sampler of Riemann-Finsler geometry (Cambridge) (D. Bao,
R. Bryant, S.S. Chern, and Z. Shen, eds.), Cambridge University
Press, 2004, pp. 1--49.


\bibitem{BCC} P. Baras  and L. Cohen, \textsl{Complete blow-up after Tmax for the solution of a semilinear heat
equation}. J. Funct. Anal. \textbf{71} (1987), 142--174.




%\bibitem{BFT} G. Barvatis, S. Filippas, A. Tertikas, \textsl{A unified approach to improved $L^p$ Hardy inequalities with best constants}, Trans. Amer. Math. Soc.
%\textbf{356} (2003), no. 6, 2169--2196.


%\bibitem{BC} S. B\'acs\'o, X. Cheng and Z. Shen, \textsl{Curvature properties of ($\alpha, \beta$)-metrics},
%Advanced Studies in Pure Mathematics, Math. Soc. of Japan, \textbf{48}(2007),
%73-110.




\bibitem{BCS} D. Bao, S. S. Chern and Z. Shen, \textsl{An Introduction
	to Riemannian-Finsler Geometry}. GTM {\bf{200}}, Springer-Verlag,
2000.

%\bibitem{BRS} D. Bao, C. Robles, Z. Shen,
%\textit{Zermelo navigation on Riemannian manifolds.}
%J. Differential Geom. \textbf{66} (2004), no. 3, 377--435.

%\bibitem{BFT} G. Barbatis, S. Filippas, A. Tertikas, \textsl{A unified approach to improved $L^p$ Hardy inequalities with
%	best constants}. Trans. Amer. Math. Soc. \textbf{356} (2004), 2169--2196.


\bibitem{BGD} E. Berchio, D. Ganguly, G. Grillo,  \textsl{Sharp Poincar\'e-Hardy and Poincar\'e-Rellich inequalities on the hyperbolic space.} J. Funct. Anal. \textbf{272} (2017), no. 4, 1661--1703.


%\bibitem{BW2} L. Berwald, \textsl{Parallel\"ubertragung in allgemeinen R\"aumen}, Atti Congr. Intern. Mat. Bologna \textbf{4}(1928), 262-270.


\bibitem{BV} H. Brezis, J.-L. V\'azquez, \textsl{Blow-up solutions of some nonlinear elliptic problems}. Revista Mat.
Univ. Complutense Madrid \textbf{10} (1997), 443--469.




\bibitem{Ca} G. Carron, \textsl{In\'egalit\'es de Hardy sur les vari\'et\'es riemanniennes non-compactes}. J. Math. Pures
Appl. (9) \textbf{76} (1997), no. 10, 883--891.



%\bibitem{CKN} L. Caffarelli, R. Kohn, L. Nirenberg, \textsl{First order interpolation inequalities with weight}. Compos.
%Math. \textbf{53} (1984), 259--275.


%\bibitem{Chern} S. S. Chern, \textsl{Finsler geometry
	%is just Riemannian
	%geometry without the
	%quadratic restriction}.  Notices Amer. Math. Soc. \textbf{43} (1996), no. 9, 959--963.

\bibitem{CHZ} S. S. Chern, Z. Shen, \textsl{Riemann-Finsler geometry}, World Scientific, 2005.


\bibitem{CM} X. Cabr\'e, Y. Martel,
\textsl{Existence versus explosion instantan\'ee pour des \'equations de la
chaleur lin\'eaires avec potentiel singulier}. C.R. Acad. Sci. Paris Ser. I Math. \textbf{329} (1999)
973--978.


%\bibitem{CZ1} X. Chen and Z. Shen, \textsl{Randers Metrics with Special Curvature Properties}, Osaka J. of Math., \textbf{40}(1),  87--101.

%\bibitem{CZ} X. Cheng, Z. Shen, \textsl{A class of Finsler metrics with isotropic $S$-curvature}. Israel J.  Math. \textbf{169} (2009), 317--340.

%\bibitem{Chern} S. S. Chern, \textit{Finsler geometry
%	is just Riemannian
%	geometry without the
%	quadratic restriction}.  Notices Amer. Math. Soc. \textbf{43} (1996), no. 9, 959--963.

%\bibitem{CX} M. P. do Carmo, C. Xia, \textsl{Complete manifolds with non-negative Ricci curvature and the Caffarelli-Kohn-Nirenberg inequalities}. Compos. Math. \textbf{140} (2004), 818--826.


\bibitem {IC} I. Chavel, \textsl{Riemannian Geometry: A
modern introduction}, Cambridge Univ., 1993.


\bibitem{D} L. D'Ambrosio, \textsl{
Hardy-type inequalities related to degenerate elliptic differential operators}.
Ann. Sc. Norm. Super. Pisa Cl. Sci. (5), \textbf{IV} (2005), 451--586

\bibitem{DD} L. D'Ambrosio, S. Dipierro, \textsl{Hardy inequalities on Riemannian manifolds and applications}. Ann. Inst. H. Poincar\'e Anal.
Non Lin\'eaire \textbf{31} (2014), no. 3, 449--475.


\bibitem{DA} E.-B.Davies, \textsl{A review of Hardy inequalities}. Oper. Theory Adv. Appl. \textbf{110} (1998), 55--67.





\bibitem{E} D. Egloff, \textsl{Uniform Finsler Hadamard manifolds}, Ann. Inst. Henri Poincar\'e, \textbf{66}(1997), 323-357.

%\bibitem{Er} W. Erb, \textsl{Uncertainty principles on Riemannian manifolds}. PhD Dissertation, Technical University Munchen, 2009.

%\bibitem{FKV}  C. Farkas, A. Krist\'aly, C. Varga,  \textsl{Singular Poisson equations on Finsler-Hadamard manifolds.} Calc. Var. Partial Differential Equations \textbf{54} (2015), no. 2, 1219--1241.

%\bibitem{F} C. Fefferman, \textsl{The uncertainty principle}. Bull. Amer. Math. Soc. (N.S.) \textbf{9} (1983), no. 2,
%129--206.


%\bibitem{FT} S. Filippas, A. Tertikas, \textsl{Optimizing improved Hardy inequalities}. J. Funct. Anal. \textbf{192} (2002)
%186--233.

%\bibitem{FM} P. Foulon, V.S. Matveev,  \textit{Zermelo deformation of Finsler metrics by Killing vector fields}. Electron. Res. Announc. Math. Sci. \textbf{25} (2018), 1--7.

%\bibitem{GGE} F. Gazzola, H.-C. Grunau, E. Mitidieri, \textsl{Hardy inequalities with optimal constants and remainder terms}, Trans. Amer. Math. Soc. \textbf{356} (2004), 2149--2168.




%\bibitem{GM} N. Ghoussoub, A. Moradifam, \textsl{On the best possible remaining term in the Hardy inequality}. Proc.
%Natl. Acad. Sci. USA \textbf{105} (2008), no. 37, 13746--13751.

%\bibitem{GM2} N. Ghoussoub, A. Moradifam, \textsl{Bessel pairs and optimal Hardy and Hardy-Rellich inequalities}.
%Math. Ann. \textbf{349} (2011), 1--57.


\bibitem{HPL} G. Hardy, G. P\'olya, J.E. Littlewood, \textsl{Inequalities}. 2nd edition, Cambridge University,
 1952.


\bibitem{H} E. Hebey, \textsl{Sobolev Spaces on Riemannian Manifolds}. Springer, 1996.

%\bibitem{HKM} J. Heinonen, T. Kilpel\"ainen, O. Martio, \textsl{Nonlinear potential theory of degenerate elliptic equations}, Oxford university, 1993.


\bibitem{HKZ} L. Huang, A. Krist\'aly, W. Zhao, \textsl{Sharp uncertainty principles on general Finsler manifolds}, arXiv:1811.08697, 2018 .



%\bibitem{HM2011} L. Huang, X. Mo, \textsl{On geodesics of Finsler metrics via navigation problem}. Proc. Amer. Math. Soc.
%\textbf{139} (2011) 8, 3015--3024.


%\bibitem{HX} L. Huang, Q. Xue, \textit{Affine vector fields on Finsler manifolds}. Manuscripta Math., to appear.


%\bibitem{Kt} T. Kilpel\"ainen, \textsl{Weighted Sobolev spaces and capacity}, Ann. Acad. Sci. Fenn. Ser. AI Math, \textbf{19} (1994), 95-113.


%\bibitem{K} A. Krist\'aly, \textsl{Sharp uncertainty principles on Riemannian manifolds: the influence of curvature}.
%J. Math. Pures  Appl.  (9) \textbf{119} (2018), 326--346.

%\bibitem{Kristaly-JGA} A. Krist\'aly, \textsl{A sharp Sobolev interpolation inequality on Finsler manifolds}.
%J. Geom. Anal.  \textbf{25} (2015), no. 4, 2226--2240.


%\bibitem{K2} A. Krist\'aly, \textsl{Sharp uncertainty principles on Finsler manifolds: the effect of curvature}. arXiv:1311.6418v2.


%\bibitem{KOh} A. Krist\'aly, S. Ohta, \textsl{Caffarelli-Kohn-Nirenberg inequality on metric measure spaces with applications}.
%Math. Ann. \textbf{357} (2013), no. 2, 711--726.




\bibitem{KR} A. Krist\'aly, D. Repov\v s, \textsl{Quantitative Rellich inequalities on Finsler-Hardamard manifolds}. Commun. Contemp. Math. \textbf{18} (6) (2016), p.17.
% DOI: 10.1142/S0219199716500206.

\bibitem{KRudas} A. Krist\'aly, I.  Rudas, \textsl{Elliptic problems on the ball endowed with Funk-type metrics.} Nonlinear Anal. \textbf{119} (2015), 199--208.


\bibitem{KS} A. Krist\'aly, A. Szak\'al, \textsl{Interpolation between Brezis-V\'azquez and Poincar\'e inequalities on nonnegatively curved spaces: sharpness and rigidities}. J. Differential Equations \textbf{266} (2019), no. 10, 6621--6646.




\bibitem{LS} B. Li and Z. Shen, \textsl{On projectively flat fourth root metrics}, Canad. Math.Bull, \textbf{55}(2012), no. 1, 138--145.



%\bibitem{LLL} R.T. Lewis, J. Li, Y. Li
%\textsl{A geometric characterization of a sharp Hardy inequality},
%J. Funct. Anal., \textbf{262}(2012), 3159--3185




\bibitem{KO} I. Kombe, M. \"Ozaydin, \textsl{Improved Hardy and Rellich inequalities on Riemannian manifolds}. Trans.
Amer. Math. Soc. \textbf{361} (2009), no. 12, 6191--6203.


\bibitem{KO2} I. Kombe, M. \"Ozaydin, \textsl{Hardy-Poincar\'e, Rellich and uncertainty principle inequalities on Riemannian
	manifolds}. Trans. Amer. Math. Soc. \textbf{365} (2013), no. 10, 5035--5050.


	%\bibitem{Matsumoto} M. Matsumoto, {\it A slope of a mountain is a Finsler surface with respect
	%to a time measure}. {J. Math. Kyoto Univ.} \textbf{29} (1989), 17--25.


%\bibitem{MMP} M. Marcus, V.J. Mizel, Y. Pinchover, \textsl{On the best constant for Hardy¡¯s inequality in $\mathbb{R}^n$}, Trans. Amer. Math.
%Soc. \textbf{350} (8) (1998), 3237--3255.



\bibitem{MST} A. Mercaldo, M. Sano, F. Takahshi, \textsl{Finsler Hardy inequalities}, arXiv: 1806.04901v2.

\bibitem{Ne} I. Newton, Arithmetica Universalis: Sive de Compositione et Resolutione Arithmetica Liber, 1707.
%\bibitem{MH2007} X. Mo, L. Huang, \textsl{On curvature decreasing property of a class of navigation problems}.
%Publ. Math. Debrecen. \textbf{71} (2007), 141--163.



\bibitem{Ot} S. Ohta, K.-T. Sturm, \textsl{Heat flow on Finsler manifolds}, Comm. Pure Appl. Math.
\textbf{62}(10) (2009), 1386--1433.


\bibitem{O} S. Ohta, \textsl{Finsler interpolation inequalities}, Calc. Var. Partial Differential
Equations, \textbf{36}(2009), 211--249.


\bibitem{O1} S. Ohta, \textsl{Optimal transport and Ricci curvature in Finsler geometry. Probabilistic
approach to geometry}, 323--342, Adv. Stud. Pure Math., 57, Math. Soc. Japan,
Tokyo, 2010.


\bibitem{O2} S. Ohta,\textsl{Nonlinear geometric analysis on Finsler manifolds}, Eur. J. Math.  \textbf{3}(4)(2017), 916--952.



\bibitem{PV} I. Pera, J. L. V\'azquez, \textsl{On the stability or instability of the singular solution of the
semilinear heat equation with exponential reaction term}. Arch. Rational Mech. Anal. \textbf{129}
(1995), 201--224.







\bibitem{R} H. Rademacher, \textsl{Nonreversible Finsler metrics of positive
	flag curvature}. A sampler
of Riemann-Finsler geometry, Cambridge Univ. Press, Cambridge, 2004, pp. 261--302.

%\bibitem{RS1} M. Ruzhansky, D. Suragan,  \textsl{Hardy and Rellich inequalities, identities, and sharp remainders on homogeneous groups.} Adv. Math. 317 (2017), 799--822.

%\bibitem{RS2} M. Ruzhansky, D. Suragan, \textsl{Uncertainty relations on nilpotent Lie groups.} Proc. A. 473 (2017), no. 2201, 20170082, 12 pp.

%\bibitem{RS3} M. Ruzhansky, D. Suragan,  \textsl{Layer potentials, Kac's problem, and refined Hardy inequality on homogeneous Carnot groups.} Adv. Math. 308 (2017), 483--528.

%\bibitem{BS} B. Shen, \textsl{Strongly and weakly affine vector fields on Finsler manifolds}, preprint. % Differ. Geom. Appl., to appear.

\bibitem{Shen_Adv_Math} Z. Shen, \textsl{Volume comparison and its applications in Riemann-Finsler geometry}. Adv. Math. \textbf{128} (1997), no. 2, 306--328.

\bibitem{Shen2013} Z. Shen, \textsl{Differential geometry of spray and Finsler spaces}.  Kluwer Academic
Publishers, 2001.

%\bibitem{Shen2003} Z. Shen, \textsl{Finsler metrics with $K=0$ and $S=0$}. Canadian J. Math. \textbf{55} (2003), 112--132.



\bibitem{Sh1} Z. Shen, \textsl{Lectures on Finsler geometry}. World
Sci., Singapore, 2001.








\bibitem{V} J. L. Vazquez, \textsl{Domain of existence and blowup for the exponential reaction-diffusion equation}.
Indiana Univ. Math. J. \textbf{48} (1999), 677--709.



\bibitem{VZ} J. L.  Vazque, E. Zuazua, \textsl{The Hardy inequality and the asymptotic behavior of the
heat equation with an inverse-square potential}. J. Funct. Anal. \textbf{173} (2000), 103--153.









%\bibitem{WW} Z.-Q. Wang, M. Willem, \textsl{Caffarelli-Kohn-Nirenberg inequalities with remainder terms}. J. Funct.
%Anal. \textbf{203} (2003), 550--568.

\bibitem{WW22} G. Wei, W. Wylie, \textsl{Comparison geometry for the Bakry-Emery Ricci tensor},
J. Differential Geom., \textbf{83} (2009), 377--406

\bibitem{Wu2} B. Y. Wu,   \textsl{On integral Ricci curvature and topology of Finsler manifolds}. Int. J. Math., \textbf{23}(11), (2012),
https://doi.org/10.1142/S0129167X1250111X



\bibitem{WX} B. Y. Wu and Y. L. Xin, \textsl{Comparison theorems in Finsler geometry and their applications}. Math. Ann.
\textbf{337}(2007), no. 1, 177--196.



%\bibitem{X} C. Xia, \textsl{The Caffarelli-Kohn-Nirenberg inequalities on complete manifolds}. Math. Res. Lett. \textbf{14}
%(2007), no. 5, 875--885.


\bibitem{YSK} Q. Yang, D. Su, Y. Kong, \textsl{Hardy inequalities on Riemannian manifolds with negative curvature}. Commun. Contemp.
Math. \textbf{16} (2014), https://doi.org/10.1142/S0219199713500430



\bibitem{Y} S. Yin, \textsl{Comparison theorems on Finsler manifolds
with weighted Ricci curvature bounded below}, Front. Math. China, \textbf{13}(2) (2018), 435--448.










%\bibitem{HKZ} L. Huang, A. Krist\'aly, W. Zhao, \textsl{Sharp uncertaninty principles on general Finsler manifolds},








\bibitem{YZS} L. Yuan, W. Zhao, Y. Shen, \textsl{Improved Hardy and Rellich inequalities on nonreversible Finsler manifolds}, J. Math. Anal. Appl., \textbf{458}(2018), 1512-1545.


\bibitem{ZS} W. Zhao, Y. Shen, \textsl{A universal volume comparison theorem
	for Finsler manifolds and related results}.  Can. J. Math., \textbf{65} (2013), 1401-1435.


%\bibitem{Z} W. Zhao, \textsl{A lower bound for the length of closed geodesics on a Finsler manifold}, Canad. Math. Bull. \textbf{75}(2014), 194-208.

%\bibitem{Z2} W. Zhao, \textsl{A comparison theorem for Finsler submanifolds and its applications},  arXiv:1710.10682, 2017.


\end{thebibliography}
\end{document}